\documentclass{amsart}
\usepackage[utf8]{inputenc}
\usepackage{bbold}
\usepackage{amsmath}
\usepackage{graphicx, hyperref}
\usepackage{wrapfig}
\usepackage{amsthm}
\usepackage{color}
\usepackage{float}
\usepackage{tikz}
\usetikzlibrary{cd, arrows.meta}

\usepackage[utf8]{inputenc}
\usepackage{amsmath, color}
\usepackage{graphicx}
\usepackage{wrapfig}
\usepackage{amsthm}
\usepackage[margin=1in]{geometry}

\definecolor{lightviolet}{RGB}{197,180,227}
\definecolor{darkcyan}{RGB}{0, 170, 170}

\DeclareMathOperator{\indeg}{indeg}
\DeclareMathOperator{\odeg}{outdeg}

\DeclareMathOperator{\Ver}{V}
\DeclareMathOperator{\Edg}{E}
\DeclareMathOperator{\init}{init}
\DeclareMathOperator{\fin}{fin}
\DeclareMathOperator{\var}{var}
\DeclareMathOperator{\wt}{wt}

\newtheorem{theorem}{Theorem}[section]
\newtheorem*{recalltheorem}{Theorem~\ref{MatroidVersion}}
\newtheorem*{thm2}{Theorem~\ref{thm2}}
\newtheorem*{best}{Theorem~\ref{best}}
\newtheorem*{ourbest}{Theorem~\ref{thm:ourbest}}
\newtheorem*{recallcorollary}{Corollary~\ref{cor:alexpoly}}
\newtheorem*{recallconjecture}{Conjecture~\ref{conj:log-concave}}

\newtheorem{lemma}[theorem]{Lemma}
\newtheorem{proposition}[theorem]{Proposition}
\newtheorem{definition}[theorem]{Definition}

\newtheorem{corollary}[theorem]{Corollary}

\newtheorem{conjecture}[theorem]{Conjecture}
\newtheorem{remark}[theorem]{Remark}

\newtheorem{notation}[theorem]{Notation}

\title{On the Alexander polynomial of special alternating links}\date{}

\author{Elena S.~Hafner}

\address{Elena S.~Hafner, Department of Mathematics, University of Washington, Seattle, WA 98195. \newline\textup{eshafner@uw.edu}
}

\author{Karola M\'esz\'aros}

\address{Karola M\'esz\'aros, Department of Mathematics, Cornell University, Ithaca, NY 14853. \newline\textup{karola@math.cornell.edu}
}

\author{Alexander Vidinas}

\address{Alexander Vidinas, Department of Mathematics, Cornell University, Ithaca, NY 14853. \newline\textup{acv42@cornell.edu}
}

\thanks{Karola M\'esz\'aros received support from  NSF Grants DMS-1847284  and  ~DMS-2348676.    Elena S.~Hafner received support from NSF Grants DMS-1847284 and DMS-2402150.  Alexander Vidinas received support from NSF Grant DMS-1847284. }

\begin{document}
\maketitle

\begin{abstract} The Alexander polynomial (1928) is the first polynomial invariant of links devised to help distinguish links up to isotopy. Fox's conjecture (1962) -- stating that the absolute values of the coefficients of the Alexander polynomial for any alternating link are  trapezoidal -- was settled for special alternating links by the present authors (2023); K\'alm\'an, the second author, and Postnikov gave an alternative proof (2025). The present paper is a study of the special combinatorial and discrete geometric properties that Alexander polynomials of special alternating links possess along with a generalization to all Eulerian graphs, introduced by Murasugi and Stoimenow (2003). We prove that the Murasugi and Stoimenow generalized Alexander polynomials can be expressed  in terms of volumes of root polytopes of unimodular matrices. The latter generalizes a result regarding the Alexander polynomials of special alternating links that follows by putting together the work  of  Li and Postnikov (2013) and K\'alm\'an, the second author, and Postnikov (2025).  Furthermore, we conjecture a generalization of Fox's conjecture to the generalized Alexander polynomials of Murasugi and Stoimenow and  bijectively relate two longstanding combinatorial models for the Alexander polynomials of special alternating links: Crowell's state model (1959) and Kauffman's state model (1982, 2006).

\end{abstract}

\section{Introduction}

 The central question of knot theory is that of distinguishing links up to isotopy. Discovered in the 1920's \cite{alexander1928topological}, the Alexander polynomial $\Delta_L(t)$, associated to an oriented link $L$, was the first polynomial knot invariant devised for this purpose.  The key property of the Alexander polynomial is that if oriented links $L_1$ and $L_2$ are isotopic, then $\Delta_{L_1}(t)\sim\Delta_{L_2}(t)$ where $\sim$ denotes equality up to multiplication by $\pm t^k$ for some integer ~$k$.
 
 \medskip
 
A tantalizing conjecture of Fox from 1962 \cite{fox} regarding the coefficients of the Alexander polynomials of alternating links remains stubbornly open to this day. It states that the absolute values of the coefficients of the Alexander polynomial of any alternating link are trapezoidal.   Recall that a sequence $a_1, \ldots, a_n$ is \textbf{trapezoidal} if  $$a_1<a_2<\cdots <a_k=a_{k+1}=\cdots=a_m>a_{m+1}>\cdots>a_n$$ for some $1\leq k\leq m\leq n$.  For any alternating link $L$, \cite{Crowell, murasugi1958alexander, murasugi1958genus2} show that the Alexander polynomial $\Delta_L(t)$ can be normalized such that $\Delta_L(-t)\in \mathbb{Z}_{\geq 0}[t]$.  In particular, the coefficients of $\Delta_L(t)$ have alternating signs, so Fox's conjecture can be restated in the following way:

\begin{conjecture} \label{fox} \cite{fox} Let $L$ be an alternating  link. Then the coefficients of $\Delta_L(-t)$ form a trapezoidal sequence.

\end{conjecture}

 Stoimenow \cite{stoi} strengthened Fox's conjecture to log-concavity without internal zeros. Fox's conjecture remains open in general although some special cases have been settled  by Hartley \cite{H79} for two-bridge knots,  Murasugi \cite{murasugi} for a family of alternating algebraic links, and Ozsv\'ath and Szab\'o \cite{OS03} for the case of genus $2$ alternating knots. Jong \cite{jong2009alexander} also confirmed that Fox's conjecture holds for genus $2$ alternating knots.  The authors of the current paper  resolved Fox's conjecture for special alternating links \cite{hafner2023logconcavity} in 2023; Azarpendar, Juh\'asz, and K\'alm\'an \cite{tamas-recent} built on these results to resolve Fox's conjecture for certain diagrammatic Murasugi sums of special alternating links in 2024; while  K\'alm\'an, the second author, and Postnikov \cite{trap} gave an alternative proof of the special alternating case with a matroidal view in 2025.

 \medskip

The methods used in  \cite{hafner2023logconcavity} are specific to the deeply combinatorial features of the family of special alternating links and do not in an immediate fashion extend to all alternating links. 
The combinatorial and discrete geometric properties of special alternating links are also on display in the beautiful works of K\'alm\'an and Murakami \cite{murakami} and  K\'alm\'an and Postnikov \cite{KTP2017} in which the top of the HOMFLY polynomial of special alternating links is expressed in terms of the $h$-vector of triangulations of root polytopes. 

\medskip

The present paper is devoted to the study of the combinatorics and discrete geometry of the Alexander polynomials of special alternating links as well as their generalization to all Eulerian digraphs as defined by Murasugi and Stoimenow \cite{even}. As it turns out, a classical theorem in graph theory by de Bruijn, Ehrenfest \cite{best1}, Smith, and Tutte \cite{best2} which states that the number of 
 oriented spanning trees of an Eulerian digraph is independent of the choice of root is deeply ingrained in this story, and we start by explaining why.  First, however, we provide the construction of a special alternating link as it is the object that governs our pursuits.

\subsection{Notation and graph theory definitions}
\label{sec:notation}
 For any graph $G$, let $\Ver(G)$ denote the vertex set of $G$ and $\Edg(G)$ the edge set of $G$.  For a digraph $D$, we will denote the initial vertex of any $e \in \Edg(D)$ by $\init(e)$ and the final vertex of $e$ by $\fin(e)$.  The number of edges with initial vertex $v$ will be denoted $\odeg(v)$, and the number of edges with final vertex $v$ will be denoted $\indeg(v)$. We say $D$ is \textbf{Eulerian} if it is connected, and at each vertex $v\in \Ver(D)$, $\indeg(v)=\odeg(v)$.  A planar Eulerian digraph is called an \textbf{alternating dimap} if its planar embedding is such that, reading cyclically around each vertex, the edges alternate between incoming and outgoing.
 
 \subsection{Special alternating links from planar Eulerian digraphs} \label{sec:SpecialAltDef} 

There are two types of crossings in a link diagram: positive and negative.
\begin{center} \begin{tikzpicture}
\draw[black, thick] (1,0) -- (1,.9);
\draw[-stealth, black, thick] (1,1.1) -- (1,2);
\draw[-stealth, black, thick] (0,1) -- (2,1);
\draw[black, thick] (4,1) -- (4.9,1);
\draw[-stealth, black, thick] (5.1,1) -- (6,1);
\draw[-stealth, black, thick] (5,0) -- (5,2);

\node at (3,1.2) {or};
\end{tikzpicture} \end{center} 
Crossings of the form on the left are called \textbf{positive crossings} and those of the form on the right are called \textbf{negative crossings}.

 A \textbf{special alternating link} is any link which arises from the following construction, as presented by Juh\'asz, K\'alm\'an and Rasmussen \cite{jkr} and  by K\'alm\'an and Murakami \cite{murakami}.  Let $G$ be a planar bipartite graph, and let $D$ be its planar dual.  Observing that $D$ is Eulerian, we orient its edges so that $D$ is an alternating dimap.  Let $M(G)$ be the \textbf{medial graph} of $G$: the vertex set of $M(G)$ is the set $\{v_e \ | \ e\in\Edg(G)\}$, and two vertices $v_e$ and $v_{e'}$ of $M(G)$ are connected by an edge if the edges $e$ and $e'$ are consecutive in the boundary of a  face of $G$. We think of a particular planar drawing of $M(G)$ here:  the midpoints of the edges of the planar drawing of $G$ are the vertices of $M(G)$. Thinking of $M(G)$ as a flattening of a link, the orientation of $D$ will determine how we orient $M(G)$.  Partition the regions of $M(G)$ corresponding to vertices of $G$ into color classes $A$ and $B$ such that whenever an edge of $D$ meets a vertex of $M(G)$, we pass a region in $A$ on the left and a region in $B$ on the right when traveling along the edge in the direction of its orientation.  We orient the edges of $M(G)$ such that the boundary of every region in $A$ (resp. $B$) is a counterclockwise (resp. clockwise) cycle and replace the vertices of $M(G)$ with over- and under-crossings such that either every crossing is positive or every crossing is negative.  The case where every crossing is positive is called a \textbf{positive special alternating link}, and the case where every crossing is negative is a \textbf{negative special alternating link}.  We denote by $L_G$ the positive special alternating link arising from a planar bipartite graph $G$ in this way.  See Figure \ref{fig:bipandlink} for an example.

\begin{figure}
\centering
\includegraphics[height=5.75cm]{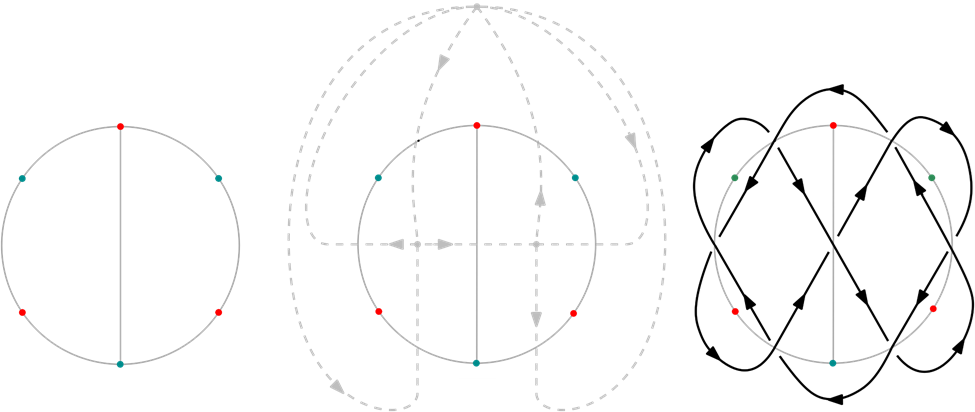}
\caption{A planar bipartite graph $G$ along with its dual alternating dimap $D$ and its associated positive special alternating link $L_G$.}
\label{fig:bipandlink}
\end{figure}

 \subsection{The BEST Theorem for Eulerian digraphs} \label{sub:best}
Recall that an \textbf{Eulerian tour} of $D$ is a directed cycle that uses each edge of $D$ exactly once. An \textbf{oriented spanning tree} rooted at $r \in \Ver(D)$ is a spanning tree of $D$ such that for each vertex $v \in \Ver(D)$, there is a unique directed path from $v$ to $r$.   
 
A classical theorem in graph theory  by de Bruijn, Ehrenfest \cite{best1}, Smith, and Tutte \cite{best2}  (named  the BEST Theorem for its authors) states that the number of Eulerian tours of an Eulerian digraph $D$ equals the number of oriented spanning trees of $D$ rooted at an arbitrary fixed vertex $r$ of $D$ multiplied by a constant dependent on $D$.  In particular, the BEST Theorem implies that:

\begin{theorem}[\cite{best1,best2}] \label{best} The number of oriented spanning trees of an Eulerian digraph  $D$ is independent of the choice of root. \end{theorem}
 
A \textbf{$k$-spanning tree} rooted at $r \in \Ver(D)$ is a spanning tree of $D$ such that reversing the orientation of exactly $k$ of its edges yields an oriented spanning tree rooted at $r$. With this terminology, $0$-spanning trees coincide with oriented spanning trees.  Analogously, an \textbf{arborescence} rooted at $r$ is a spanning tree such that for every $v \in \Ver(D)$, there is a unique directed path from $r$ to $v$, i.e. arborescences of $D$ are precisely $\{|\Ver(D)|-1\}$-spanning trees.   
\medskip

Denote by $c_k(D,r)$ the number of $k$-spanning trees rooted at $r$ in the Eulerian digraph $D$.  Inspired by the Alexander polynomial of special alternating links, Murasugi and Stoimenow \cite{even} prove the following extension of Theorem \ref{best}:

\begin{theorem}\label{thm:ourbest} \cite{even} For any $k \in \mathbb{N}$, the number of $k$-spanning trees of an Eulerian digraph  $D$ is independent of the choice of root. Moreover, the sequence $c_0(D,r)$, $\ldots$, $ c_{|\Ver(D)|-1}(D,r)$ is palindromic. Equivalently: $c_k(D,r)=c_{|\Ver(D)|-1-k}(D,r)$ for all $k\in \{0, 1, \ldots, |\Ver(D)|-1\}$.
\end{theorem}
 
 In Section \ref{sub:root}, we will provide a new proof of Theorem \ref{thm:ourbest} which will help us connect the Alexander polynomials of special alternating links to volumes of root polytopes. 
 
 \subsection{The Alexander polynomial of an Eulerian digraph} \label{sub:thm2}
Murasugi and Stoimenow \cite{even} define the \text{Alexander polynomial of an Eulerian digraph} $D$ (and arbitrary $r \in \Ver(D)$) to be:     \begin{equation} \label{pdt} P_D(t)=\sum_{k=0}^{\infty} c_k(D, r)t^k.\end{equation}

The name is justified by Murasugi and Stoimenow's theorem that for an alternating dimap $D$, the polynomial  $P_D(t)$ equals the Alexander polynomial of a special alternating link associated to $D$. Recall that an \text{alternating dimap} is a planar Eulerian digraph such that, reading cyclically around each vertex, the edges alternate between coming in and going out.

\begin{theorem}\cite[Theorem 2]{even} \label{thm2} For an alternating dimap $D$, the  polynomial  $P_D(t)$ equals the Alexander polynomial $\Delta_{L_{G}}(-t)$ for the special alternating link  $L_{G}$ associated to the planar dual $G$ of $D$, up to multiplication by $\pm t^k$ for some integer ~$k$. 
\end{theorem}
 
 We provide a proof of Theorem \ref{thm2} in Section \ref{sec:thm2} which we will utilize in Section \ref{sec:bij} to relate the Kauffman and Crowell state models. 
 We note that  
Murasugi and Stoimenow denote the polynomial $P_D$  by $\Delta_G$ as they start with an undirected graph $G$ of even valence and then put an arbitrary Eulerian orientation on it. For us, the notation $P_D$ will be more convenient. 

\medskip

 Previous work of the authors \cite[Theorem 1.2]{hafner2023logconcavity} proves that for alternating dimaps $D$, the  Alexander polynomial $\Delta_{L_{G}}(-t)$ for the special alternating link  $L_{G}$ associated to the planar dual $G$ of $D$ is log-concave.   In light of Theorem \ref{thm2}, this result can be rephrased as follows for the  polynomial ~$P_D$:

\begin{theorem}\cite[Theorem 1.2]{hafner2023logconcavity} \label{ourlogconc} For any alternating dimap $D$, the coefficients of the  polynomial $P_D(t)$ form a log-concave sequence with no internal zeros.
\end{theorem}

We 
  conjecture that Theorem \ref{ourlogconc} generalizes as follows:

\begin{conjecture}
\label{conj:log-concave}  
   For any Eulerian digraph $D$, the coefficients of the  polynomial $P_D(t)$ form a log-concave sequence with no internal zeros.
\end{conjecture}
  
  In Section \ref{sec:conj}, we prove that Conjecture \ref{conj:log-concave} holds for a special family of (not necessarily planar) Eulerian digraphs.

\subsection{The polynomial $P_D$ in terms of root polytopes} \label{sub:root} Root polytopes of bipartite graphs were defined and extensively studied by Postnikov \cite{P05}. They were related to special alternating links in the works of K\'alm\'an and Murakami \cite{murakami} and  K\'alm\'an and Postnikov \cite{KTP2017}. T\'othm\'er\'esz \cite{T22} recently generalized Postnikov's definition of a root polytope to all oriented co-Eulerian matroids (see Section \ref{sec:root} for definitions). 
Using this generalization, we connect the polynomial $P_D(t)$, as defined in equation \eqref{pdt}, to volumes of root polytopes of oriented co-Eulerian matroids:

\begin{theorem}
    \label{MatroidVersion}
    Let $D$ be an Eulerian digraph, and let $M$ be the oriented graphic matroid associated to $D$.  Let $A$ be a totally unimodular matrix representing $M^*$, the oriented dual of $M$, and let $m$ be the size of a basis of $M^*$.  Then 
\begin{equation} \label{eq} P_D(t) = \sum_{H \text{ has property *}} Vol(Q_{H})(t-1)^{\#col(H)-m} \end{equation}
where a matrix $H$ has property * if it is obtained by deleting a set of columns from $A$ without decreasing the rank of the matrix.
\end{theorem}

Here $Q_H$ is the root polytope of the matrix $H$ (see Definition \ref{def:rootpolytope}). An immediate corollary of the above theorem is:

\begin{corollary} \label{cor:alexpoly} The Alexander polynomial $\Delta_{L_{G}}(-t)$ of the special alternating link associated to the planar bipartite graph $G$ can be expressed as:
\begin{equation} \label{lipost} \Delta_{L_{G}}(-t) \sim \sum_{\substack{H \subset G \\ H \text{ connected}}} Vol(Q_{H})(t-1)^{|\Edg(H)|-|\Ver(G)|+1}.\end{equation}
\end{corollary}

We note that Corollary \ref{cor:alexpoly} can also be deduced putting together the work of Li and Postnikov \cite{slicing} and K\'alm\'an, M\'esz\'aros, and Postnikov \cite{trap}. Indeed,  the work of Li and Postnikov \cite{slicing} was the original inspiration for our Theorem \ref{MatroidVersion}. 
The right-hand side of equation \eqref{lipost}  was introduced by Li and Postnikov \cite{slicing} in the context of computing volumes of slices of zonotopes. They proved the equation:

\begin{equation}\label{eq:lipost} f_{M_{\vec{G}}}(t)=\sum_{\substack{H \subset G \\ H \text{ connected}}} Vol(Q_{H})(t-1)^{|\Edg(H)|-|\Ver(G)|+1}\end{equation} for a polynomial $f_{M_{\vec{G}}}(t)$ associated to the oriented graphic matroid of a planar bipartite graph with orientation $\vec{G}$ where the edges are directed from one bipartition to the other. The details of the definition of this polynomial can be found in \cite{slicing} or in \cite{trap}. In turn,  K\'alm\'an, M\'esz\'aros, and Postnikov \cite{trap} prove in their Theorem 5.5 that:

\begin{equation} \label{eq:trap} \Delta_{L_{G}}(-t)\sim f_{M_{\vec{G}}}(t).\end{equation}

Putting equations \eqref{eq:lipost} and \eqref{eq:trap} together, one readily obtains  Corollary \ref{cor:alexpoly}.

 \subsection{On the Kauffman and Crowell state models of the Alexander polynomial} 
  The Alexander polynomial was initially defined by James Alexander II as a determinant \cite{alexander1928topological}. Crowell, in 1959, gave the first combinatorial state model for Alexander polynomials of alternating links \cite{Crowell}. In John Conway's paper \cite{conway1970enumeration}, he introduces the Conway polynomial which satisfies a skein relation and is related to the Alexander polynomial by a substitution. In 1982, Kauffman defined the notion of \textit{states of a link universe}, obtained from a link diagram. Originally, these states were used to compute a polynomial that could be specialized to the Conway polynomial. Later, in the most recent edition of \cite{K06}, released in 2006, Kauffman described a way of using this combinatorial state model to directly compute the Alexander polynomial of any link. 
  
  \medskip
  
Theorem \ref{thm2}, Murasugi and Stoimenow's result expressing $P_D$ in terms of $k$-spanning trees for alternating dimaps $D$, is based on Kauffman's model. We can view Theorem \ref{thm2} as another way of expressing the Kauffman states in the special alternating link case. On the other hand, the present authors work establishing Fox's conjecture for special alternating links, as in Theorem \ref{ourlogconc}, is based on and heavily inspired by the Crowell state model for alternating links. It is, thus, natural to wonder about a weight-preserving combinatorial bijection 
   between Crowell's and Kauffman's models. In Theorem \ref{maintheorem}, we construct such a bijection between the two models in the case of special alternating links.

   \medskip
   
\noindent \textbf{Roadmap of the paper.} 
In Section \ref{sec:k}, we study the enumeration of $k$-spanning trees of Eulerian digraphs which we use in Section \ref{sec:root} to prove Theorem \ref{MatroidVersion}. In Section \ref{sec:thm2}, we show that for alternating dimaps $D$, $P_D(t)$ is the Alexander polynomial of a special alternating link. The proof we give is inspired by Murasugi and Stoimenow's proof but is written in a graph theoretic language.  In Section \ref{sec:conj}, we discuss  possible generalizations of Fox's long-standing conjecture that the sequence of absolute values of coefficients of the Alexander polynomial of an alternating link is trapezoidal.  We conclude in Section \ref{sec:bij} by constructing an explicit weight-preserving bijection between the Crowell and Kauffman states of any special alternating link.

\medskip

\noindent Instead of having a lengthy background section, we will include the necessary background throughout the paper.

\section{On $k$-spanning trees of Eulerian digraphs}
\label{sec:k}
In this section, we study $k$-spanning trees of Eulerian digraphs, as defined in Section \ref{sub:best}. We will leverage the results from this section in Section \ref{sec:root} in order to connect the Alexander polynomials of Eulerian digraphs to volumes of root polytopes of oriented co-Eulerian matroids.

\medskip

 Recall from the introduction that $c_k(D,r)$ denotes the number of $k$-spanning trees rooted at $r$ in the  Eulerian digraph $D$.  We will give an alternative proof of the following theorem of Murasugi and Stoimenow:

\begin{ourbest}(\cite[Proposition 1]{even})
 For any $k \in \mathbb{N}$, the number of $k$-spanning trees of an Eulerian digraph  $D$ is independent of the choice of root. Moreover, the sequence $c_0(D,r)$, $\ldots$, $ c_{|\Ver(D)|-1}(D,r)$ is palindromic. Equivalently: $c_k(D,r)=c_{|\Ver(D)|-1-k}(D,r)$ for all $k\in \{0, 1, \ldots, |\Ver(D)|-1\}$.
\end{ourbest}

To prove Theorem \ref{thm:ourbest}, we make use of the BEST Theorem (see Section \ref{sub:best}) and some of its consequences. 

\begin{theorem} (\cite[Theorem 10.2]{stanley2013algebraic},  \cite{best2, best1})
\label{ToursToTrees}
    Let $D$ be an Eulerian digraph, and fix any edge $e \in \Edg(D)$ such that $\init(e)=r$.  Let $\epsilon(D,e)$ denote the number of Eulerian tours of $D$ beginning with $e$.  Then
    \[\epsilon(D,e)=c_0(D,r)\prod_{u\in \Ver(D)}(\odeg(u)-1)!\]
\end{theorem}

\begin{best} (\cite[Corollary 10.3]{stanley2013algebraic}, \cite{best2, best1})
\label{indepRoot0}
For any Eulerian digraph $D$, $c_0(D,r)$ is independent of the root $r \in \Ver(D)$.
\end{best}

Let $D^T$ denote the \textbf{transpose} of $D$:  the digraph obtained by reversing the direction of each edge in $D$. 

\begin{corollary}
\label{flip0}
For any Eulerian digraph $D$ and $r \in \Ver(D)$, $c_0(D,r)=c_0(D^T,r)$.  In particular, $c_0(D,r)=c_{|\Ver(D)|-1}(D,r)$.
\end{corollary}
\begin{proof}
Observe that the Eulerian tours of $D^T$ are precisely the Eulerian tours of $D$ with the edges listed in reverse.  In particular, if $e \in \Edg(D)$ has initial vertex $r$, then $\epsilon(D,e)=\epsilon(D^T,f)$ for any $f \in \Edg(D^T)$ with $\init(f)=r$.  Theorem \ref{ToursToTrees} therefore implies that $c_0(D,r)=c_0(D^T,r)$.  

The second statement follows by observing that $0$-spanning trees of $D$ (oriented spanning trees) are precisely $(|\Ver(D)|-1)$-spanning trees (arborescences) of $D^T$ and vice-versa.
\end{proof}

We note that Corollary \ref{flip0} also follows from  \cite[Theorem 4.9]{T22}.

For an arbitrary digraph $D$ and subset $E' \subset \Edg(D)$, let $D \setminus E'$ denote the graph obtained from $D$ by deleting the edges in $E'$, and let $D/E'$ denote the graph obtained from $D$ by contracting the edges in $E'$.  A \textbf{minor} of $D$ is any graph which can be obtained from $D$ by deleting and/or contracting edges.

\begin{theorem}
\label{CountingArbs}
    Let $D$ be an Eulerian digraph, and fix any $r \in \Ver(D)$.  Then
    \[c_k(D,r)=\sum_{i=0}^{k}(-1)^i\sum_{\substack{\text{\emph{cycle-free} } E' \subset \Edg(D) \\ |E'|=k-i}} \binom{|\Ver(D)|-1-(k-i)}{i}c_0(D/E',r) \]
\end{theorem}

\begin{proof}
    Observe that if $T$ is a $0$-spanning tree in $D/E'$ for some cycle-free $E' \subset \Edg(D)$ with $|E'|=k$, then $T_D:=(\Ver(D),\Edg(T) \cup E')$ is an $i$-spanning tree in $D$ rooted at $r$ for some $0 \leq i\leq k$. Similarly, if $T_D$ is an $i$-spanning tree in $D$ and $\Edg_{\text{away}}(T_D)$ is the set of edges of $T_D$ directed away from the root, then $T_D/E'$ is a $0$-spanning tree in $D/E'$ for any cycle-free $E' \supset \Edg_{\text{away}}(T_D/E')$.  We can thus write 
    \[c_k(D,r)= \sum_{\substack{\text{cycle-free } E' \subset \Edg(D) \\ |E'|=k}}c_0(D/E',r)  - \sum_{\substack{\text{cycle-free } E' \subset \Edg(D) \\ |E'|=k}} c_{< k}(D,E',r)   \]
    where $c_{< k}(D,E',r)$ is the number of $0
    $-spanning trees $T$ in $D/E'$ such that there exists $e \in E'$ where $\Edg(T) \cup \{e\}$ forms a 0-spanning tree in $D/(E'-\{e\})$.  Using inclusion-exclusion, we rewrite this as follows:
    \[c_k(D,r)= \sum_{\substack{\text{cycle-free } E' \subset \Edg(D) \\ |E'|=k}}c_0(D/E',r)  - \sum_{\substack{\text{cycle-free } E' \subset \Edg(D) \\ |E'|=k}} \sum_{i=1}^k \sum_{\substack{F \subset E' \\ |F|=i}} (-1)^{i-1} c_0(D/(E'-F),r).\]
    Observe that for any subset $F' \subset \Edg(D)$ with $|F'|=k-i$ and any $0$-spanning tree $T$ in $D/F'$, contracting any $i$ edges of $T$ produces a $0$-spanning tree of $D/E'$ for some $E'$ with $|E'|=k$.  Relabeling $F':=E'-F$ in the above equation thus yields
    \[c_k(D,r)= \sum_{\substack{\text{cycle-free } E' \subset \Edg(D) \\ |E'|=k}}c_0(D/E',r)  - \sum_{i=1}^k (-1)^{i-1}\sum_{\substack{\text{cycle-free } F' \subset \Edg(D) \\ |F'|=k-i}} \binom{|\Ver(D/F')|-1}{i} c_0(D/F',r).\]
    Since $|\Ver(D/F')|=|\Ver(D)|-(k-i)$, we can relabel and combine the sums to obtain
    \[c_k(D,r)=\sum_{i=0}^{k}(-1)^i\sum_{\substack{\text{cycle-free } E' \subset \Edg(D) \\ |E'|=k-i}} \binom{|\Ver(D)|-1-(k-i)}{i}c_0(D/E',r). \]
\end{proof}

\begin{theorem}(\cite{even})
\label{ArbsIndep}
    For any Eulerian digraph $D$, $c_k(D,r)=c_k(D^T,r)$.
\end{theorem}
\begin{proof}
    Corrolary \ref{flip0} implies that this holds when $k=0$.  Since Theorem \ref{CountingArbs} expresses $c_k(D,r)$ entirely in terms of $c_0(H,r)$ for Eulerian minors $H$ of $D$, this result follows immediately.
\end{proof}
\begin{ourbest}(\cite[Proposition 1]{even})
    For any Eulerian digraph $D$, $c_k(D,r)$ is independent of the root $r$, and the sequence $c_0(D,r)$, $\ldots$, $ c_{|\Ver(D)|-1}(D,r)$ is palindromic.
\end{ourbest}
\begin{proof}
    Since $k$-spanning trees of $D$ are precisely $(|\Ver(D)|-1-k)$-spanning trees of $D^T$, this follows immediately from Theorems \ref{best} and \ref{ArbsIndep}.
\end{proof}

\section{$P_D(t)$ in terms of volumes of  Root Polytopes of oriented co-Eulerian matroids}
\label{sec:root}

In this section, we prove Theorem \ref{MatroidVersion}.   Before we do so, we review the   work of T\'othm\'er\'esz \cite{T22} on root polytopes of {co-Eulerian matroids}, following parts of her exposition.

\medskip

 A matroid $M$ is said to be \textbf{regular} if it is representable by a \textbf{totally unimodular matrix}, i.e. a matrix whose subdeterminants are all either $1$, $0$, or $-1$. Let $A$ denote such a matrix and $\{a_1, \dots, a_m\}$ its columns. For each circuit $C=\{i_1,\dots,i_j\}$ of a regular matroid with a corresponding linear dependence relation $\sum_{k=1}^j\lambda_k a_{i_k}=0$ of columns of $A$, we may partition its elements into two sets: $C^+=\{i_k \  | \ \lambda_k > 0\}$ and $C^-=\{i_k \ | \ \lambda_k < 0\}$. Observe that scaling the expression $\sum_{k=1}^j\lambda_k a_{i_k}$ by a nonzero constant potentially interchanges $C^+$ and $C^-$, but the partition of $C$ into these two sets is well-defined up to this exchange. In this way, regular matroids can be \textbf{oriented}.

Two regular oriented matroids $M_1$ and $M_2$ on groundset $E$ have \textbf{mutually orthogonal signed circuits} if for each pair of signed circuits $C_1={C_1}^+\sqcup {C_1}^-$ and $C_2={C_2}^+\sqcup {C_2}^-$ of $M_1$ and $M_2$, respectively, either $C_1\cap C_2 = \emptyset$, or $({C_1}^+ \cap {C_2}^+)\cup({C_1}^- \cap {C_2}^-)$ and $({C_1}^+ \cap {C_2}^-)\cup({C_1}^- \cap {C_2}^+)$ are both nonempty. Every regular oriented matroid $M$ admits a unique \textbf{dual oriented matroid} $M^*$ such that $M$ and $M^*$ have mutually orthogonal signed circuits.

\begin{definition}[\cite{T22} Definition 3.4] 
    A regular oriented matroid is \textbf{co-Eulerian} if for each circuit $C$, $|C^+| = |C^-|$. 
\end{definition}

\begin{lemma}[\cite{T22} Claim 3.5]
For an Eulerian digraph $D$, the oriented dual $M^*$ of the graphic matroid of $D$ is a co-Eulerian regular oriented matroid.
\end{lemma}

With this foundation, T\'othm\'er\'esz introduces the following definition and results.

\begin{definition}[\cite{T22} Definition 3.1]
    \label{def:rootpolytope}
    Let $A$ be a totally unimodular matrix with columns $a_1, \dots, a_m$. The \textbf{root polytope} of $A$ is the convex hull $\mathcal{Q}_A := \emph{conv}(a_1, \dots, a_m)$.
\end{definition}

T\'othm\'er\'esz demonstrates that for any pair of totally unimodular matrices $A$ and $A'$ representing the same regular oriented matroid $M$, there is a lattice point-preseving linear bijection between $t\cdot\mathcal{Q}_A$ and $t\cdot\mathcal{Q}_{A'}$ for any $t\in\mathbb{Z}$.  The definition of a root polytope can thus be extended to regular oriented matroids by considering the root polytope of a totally unimodular matrix representing the matroid.  The set of arborescences of an Eulerian digraph yields a triangulation of the root polytope of  the dual matroid. Recall that an arborescence   of $D$ rooted at $r\in\Ver(D)$ is a directed subgraph $A$ such that for each other vertex $v$ of $D$, there is a unique directed path in $A$ from $r$ to $v$.

\begin{proposition}[\cite{T22} Proposition 3.8]
    For a regular oriented matroid represented by a totally unimodular matrix $A$ and a basis $B=\{i_1,\dots,i_j\}$, the simplex $\Delta_B := \emph{conv}(a_{i_1},\dots,a_{i_j})$ is unimodular. That is, its normalized volume is $1$.
\end{proposition}

\begin{theorem}[\cite{T22} Theorem 4.1]
Let $D$ be an Eulerian digraph, and let $A$ be any totally unimodular matrix representing the oriented dual of the oriented graphic matroid of $D$. Let $r\in \Ver(D)$ and $\mathcal{D} = \{B\subset \Edg(D) \ | \ \Edg(D) - B \ \text{is an arborescence of} \ D \ \text{rooted at} \ r\}$. Then $\{\Delta_B \ | \ B\in\mathcal{D}\}$ is a triangulation of $\mathcal{Q}_A$.
\end{theorem}

\begin{corollary}[\cite{T22}]
\label{VolMatroidCase}
    Let $D$ be an Eulerian digraph, and let $A$ be any totally unimodular matrix representing the oriented dual of the oriented graphic matroid of $D$. Let $r\in \Ver(D)$.  Then
    $$\emph{Vol} (\mathcal{Q}_{A}) = c_0(D,r).$$
\end{corollary}

As a consequence, we may reformulate our Theorem \ref{CountingArbs} in terms of volumes of root polytopes.

\begin{corollary}
    \label{CountingArbsVol}
    For an Eulerian digraph $D$, let $H_D$ denote a totally unimodular matrix representing the oriented dual of the graphic matroid of $D$.
    $$c_k(D,r)=\sum_{i=0}^{k}(-1)^i\sum_{\substack{\text{cycle-free } E' \subset \Edg(D) \\ |E'|=k-i}} \binom{|\Ver(D)|-1-(k-i)}{i}\emph{Vol} (\mathcal{Q}_{H_{D/E'}}).$$
\end{corollary}

We are now ready to prove:

\begin{recalltheorem}
    Let $D$ be an Eulerian digraph, and let $M$ be the oriented graphic matroid associated to $D$.  Let $A$ be a totally unimodular matrix representing $M^*$, the oriented dual of $M$, and let $m$ be the size of a basis of $M^*$.  Then 
\begin{equation}\label{eq:tutte}P_D(t)=\sum_{k=0}^{\infty} c_k(D,r)t^k = \sum_{H \text{ has property *}} Vol(Q_{H})(t-1)^{\#col(H)-m} \end{equation}
where a matrix $H$ has property * if it is obtained by deleting a set of columns from $A$ without decreasing the rank of the matrix.
\end{recalltheorem}

\proof  Note that if $k>|\Ver(D)|-1$, then $c_k(D,r)=0$ since a spanning tree in $D$  has  $|\Ver(D)|-1$ edges.  Since $\#col(A)-m=|\Ver(D)|-1$ (both sides compute the size of a basis of $M$), the right-hand side also has no terms of degree greater than $|\Ver(D)|-1$.  \\
Using a binomial expansion, the desired result is equivalent to 
\[c_k(D,r)=\sum_{i=0}^{\#col(A)-k-m} \sum_{\substack{H \text{ with property *} \\ \#col(H)=k+m+i }} (-1)^i \binom{k+i}{i} Vol(Q_{H}) \]
for all $k \leq |\Ver(D)| -1$.  Since $c_k(D,r)=c_{|\Ver(D)|-1-k}(D,r)$, this is also equivalent to
\[\begin{split}
    c_k(D,r) & = \sum_{i=0}^{\#col(A)-|\Ver(D)|+k-m+1} \sum_{\substack{H \text{ with property *} \\ \#col(H)=|\Ver(D)|-1-k+m+i }} (-1)^i \binom{|\Ver(D)|-1-k+i}{i} Vol(Q_{H}) \\
    &= \sum_{i=0}^k \sum_{\substack{H \text{ with property *} \\ \#col(H)=\#col(A)-k+i }} (-1)^i \binom{|\Ver(D)|-1-k+i}{i} Vol(Q_{H})
\end{split}\]

Recall that by standard results in matroid theory, deleting elements of $M^*$ (or columns of $A$) is equivalent to contracting the corresponding elements in $M$.  Furthermore, the set of columns we remove from $A$ to form $H$ are disjoint to a basis of $M^*$ (since removing them did not decrease the rank of the matrix), so these columns form an independent set of $M$.  Since the independent sets in the graphic matroid of $D$ are precisely cycle-free subsets of edges of $D$, the desired result is equivalent to
\[c_k(D,r)=\sum_{i=0}^{k}(-1)^i\sum_{\substack{\text{cycle-free } E' \subset \Edg(D) \\ |E'|=k-i}} \binom{|\Ver(D)|-1-(k-i)}{i}\emph{Vol} (\mathcal{Q}_{H_{D/E'}}) \]
which holds by Corollary \ref{CountingArbsVol}.
\qed
\medskip

\begin{recallcorollary}   The Alexander polynomial $\Delta_{L_{G}}(-t)$ of the special alternating link associated to the planar bipartite graph $G$ can be expressed as:
\begin{equation}\label{apost} \Delta_{L_{G}}(-t) \sim \sum_{\substack{H \subset G \\ H \text{ connected}}} Vol(Q_{H})(t-1)^{|\Edg(H)|-|\Ver(G)|+1}.\end{equation}
\end{recallcorollary}

\proof
Note that for any alternating dimap $D$ whose planar dual is the bipartite graph $G$, Theorem \ref{MatroidVersion} specializes to: 
\begin{equation}\label{post} P_D(t) = \sum_{\substack{H \subset G \\ H \text{ connected}}} Vol(Q_{H})(t-1)^{|\Edg(H)|-|\Ver(G)|+1}.\end{equation}

Then Corollary \ref{cor:alexpoly} follows readily from 
 Theorem \ref{thm2} and equation \eqref{post}.
 \qed

\medskip
As explained in Section \ref{sub:root}, Theorem \ref{MatroidVersion} was inspired by the work of Li and Postnikov \cite{slicing}, and Corollary \ref{cor:alexpoly} can also be deduced from results of Li and Postnikov \cite{slicing} and K\'alm\'an, M\'esz\'aros, Postnikov \cite{trap}.

\section{$P_D(t)$ as the Alexander polynomial of special alternating links}
\label{sec:thm2}

As mentioned in the introduction, the polynomial $P_D(t)$ is the Alexander polynomial of a special alternating link when $D$ is an alternating dimap. This was established by Murasugi and Stoimenow \cite{even}. We give a proof here for completeness. Our proof is closely related to the one they provided. We will use parts of our proof of Theorem \ref{thm2} to establish our weight-preserving bijection between the Crowell and Kauffman state models in Section \ref{sec:bij}.

\begin{thm2}\cite[Theorem 2]{even}  For an alternating dimap $D$, the Alexander polynomial of $D$,  $P_D(t)$, equals the Alexander polynomial $\Delta_{L_{G}}(-t)$ for the special alternating link  $L_{G}$ associated to the planar dual $G$ of $D$. In other words:
$$\Delta_{L_{G}}(-t) \sim \sum_{k=0}^{|\Ver(D)|-1}c_k(D,r)t^{|\Ver(D)|-1-k}.$$
\end{thm2}

\medskip

Theorem \ref{thm2} is stated in \cite{even} as $\Delta_{L_G}(-t)\sim P_D(t)$. By Theorem \ref{thm:ourbest}, $P_D(t)$ is palindromic; therefore,  $\Delta_{L_G}(-t)\sim P_D(t)$ is equivalent to $\Delta_{L_{G}}(-t) \sim \sum_{k=0}^{|\Ver(D)|-1}c_k(D,r)t^{|\Ver(D)|-1-k}$, the expression in Theorem \ref{thm2}.

\medskip

Before proving Theorem \ref{thm2}, we present Kauffman's state model.

\subsection{Kauffman's state model} 
The State Summation,  introduced by Kauffman \cite{K06}, describes the Alexander polynomial as a sum over weights of decorations of a link diagram, called states \cite{K06}. We begin with a (not necessarily alternating) link diagram $L$ and marks two (arbitrary) adjacent regions.  At every crossing, we assign a weight to each nonmarked region that meets it, according to the following convention:

 \bigskip
 
\begin{center} \begin{tikzpicture}
\draw[black, thick] (1,0) -- (1,.9);
\draw[-stealth, black, thick] (1,1.1) -- (1,2);
\draw[black, thick] (0,1) -- (2,1);
\draw[black, thick] (5,0) -- (5,.9);
\draw[-stealth, black, thick] (5,1.1) -- (5,2);
\draw[black, thick] (4,1) -- (6,1);

\node at (3,1.2) {is labeled};
\node at (4.75,1.5) {$t$};
\node at (4.65,.5) {$-t$};
\node at (5.35,1.5) {$-1$};
\node at (5.3,.5) {$1$};
\end{tikzpicture} \end{center}

A \textbf{state} is a bijection between crossings in the diagram and the nonmarked regions so that each crossing is mapped to one of the four regions it meets. This bijection is presented diagrammatically, as in Figure \ref{fig:state}, by drawing a \textbf{state marker} at each crossing to show which region it is mapped to. For each state $S$, there is a signed monomial $\langle L|S\rangle$, computed as the product of the weights corresponding to the region with the state marker at each crossing. 

Kauffman defines a \textbf{black hole} to be any state marker of the form below.  

\begin{center} \begin{tikzpicture}
\draw[black, thick] (1,0) -- (1,.9);
\draw[-stealth, black, thick] (1,1.1) -- (1,2);
\draw[-stealth, black, thick] (0,1) -- (2,1);
\draw[black, thick] (4,1) -- (4.9,1);
\draw[-stealth, black, thick] (5.1,1) -- (6,1);
\draw[-stealth, black, thick] (5,0) -- (5,2);

\node at (3,1.2) {or};
\draw[-{Latex[round]}, black, thick] (.7,.7)--(.8,.8);
\draw[-{Latex[round]}, black, thick] (4.7,.7)--(4.8,.8);
\end{tikzpicture} \end{center} 

Let $b(S)$ be the number of black holes in a fixed state $S$.

\begin{figure}
\centering
\includegraphics[width=11cm]{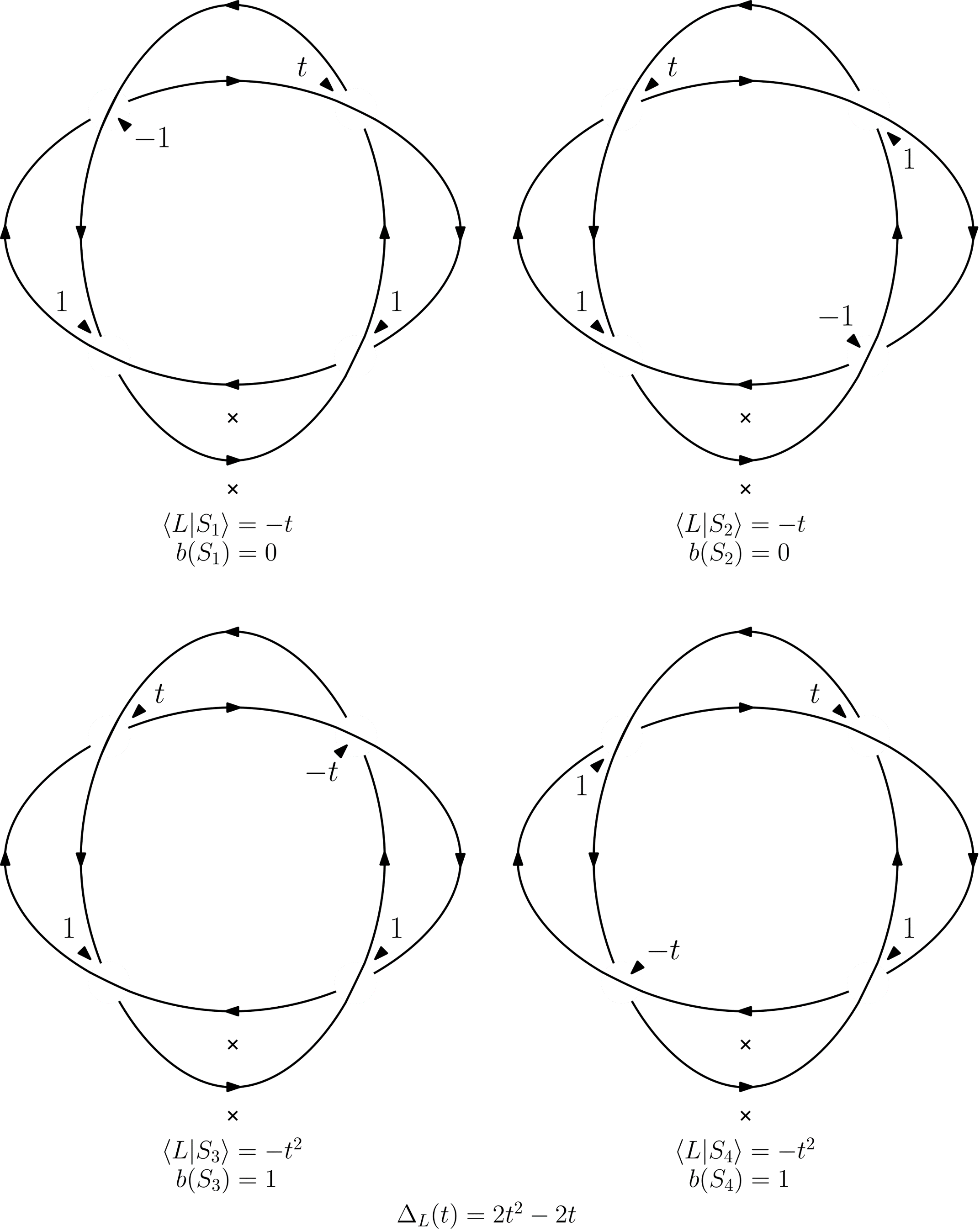}
\caption{The four states of the positive $K(2,2)$ along with their monomial weights and the resulting Alexander polynomial.  The two adjacent marked regions are each labeled with an x.}
\label{fig:state}
\end{figure}

\begin{theorem}[\cite{K06} pg. 176]
    For any link $L$, $\Delta_L(t)\sim\sum_S \langle L|S\rangle(-1)^{b(S)}$.
\end{theorem} 

Kauffman also interprets these states graph theoretically. Given any diagram of a link, 2-color its regions so that at each crossing, no two regions of the same color are adjacent. Partition the set of regions according to color into sets $V$ and $V'$ where $V'$ is the set containing the exterior. Given a link diagram, the \textbf{checkerboard graph} of the link is the graph with vertex set $V$ where two vertices $r,r'\in V$ are connected by an edge if and only if their corresponding regions meet at a crossing in the link diagram. We will denote the checkerboard graph by $G$ and its planar dual by $G^*$. Notice that $G^*$ may be constructed in the same way by taking the vertex set to be $V'$. We refer to $G^*$ as the checkerboard dual graph. 

Kauffman provides a bijection between states of a link diagram and spanning trees of the checkerboard graph (or equivalently, spanning trees of the checkerboard dual). Given a state $S$, one draws an edge between a pair of vertices in its associated tree if and only if the regions corresponding to those vertices are separated by a crossing whose state marker is in one of those two regions (see Figure \ref{fig:tree}).  In this way, a state identifies a pair of spanning trees that we refer to as \textbf{dual spanning trees}: one in $G$ and one in $G^*$. We note that these dual spanning trees are not planar duals of each other; rather, they correspond to bases in the graphic matroids of  $G$ and  $G^*$ that are complements of one another. 

Conversely, given a spanning tree $T$ of $G$ and its dual spanning tree $T'$ in $G^*$ (which consists of the edges in $\Edg(G^*)$ which do not correspond to edges in $\Edg(T)$, see Figure \ref{fig:tree}), one recovers a state by assigning edge orientations that make $T$ and $T'$ oriented spanning trees rooted at the vertices corresponding to the two adjacent marked regions. When a directed edge is incident to a crossing, place a state marker pointing into that crossing from the region corresponding to the initial vertex of the edge, as in Figure \ref{fig:tree}. 

\begin{figure}
\centering
\includegraphics[width=11cm]{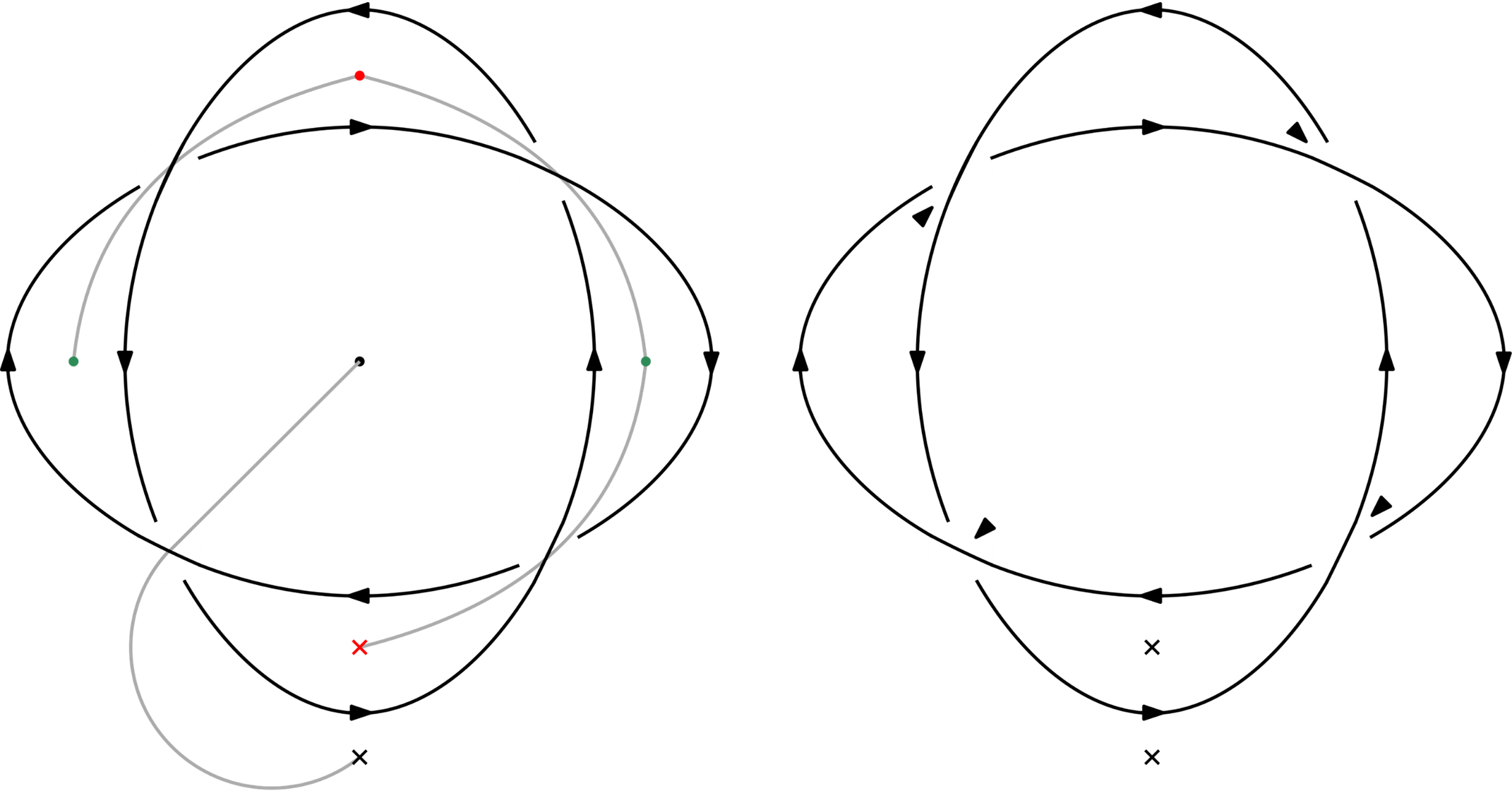}
\caption{A tree in the checkerboard graph of the positive $K(2,2)$ and its dual along with the corresponding state.}
\label{fig:tree}
\end{figure}

\subsection{Proof of Theorem \ref{thm2}}
To prove Theorem \ref{thm2}, we require the following lemma which we will also utilize in our weight-preserving bijection between the Crowell and Kauffman models.

\begin{lemma}
\label{SpanningTreesLemma}
Fix a bipartite graph $G$ with color classes $A$ and $B$, and let $r' \in \Ver(G)$. Let $S$ be any spanning tree of $G$.  Orient the edges of $S$ so that it is an oriented spanning tree  rooted at $r'$. The number of oriented edges $(a,b)\in\Edg(S)$ where $a\in A$ and $b\in B$ is independent of the choice of $S$.
\end{lemma}

\begin{proof}
Since the color classes $A$ and $B$ form a proper 2-coloring of $G$, every edge with initial vertex in $A$ has the form $(a,b)\in\Edg(S)$ where $a\in A$ and $b\in B$.  If $S$ is an oriented spanning tree, then by definition, every non-root vertex $v\in \Ver(G)$ is the initial vertex of exactly one edge of $S$.  In particular, the number of edges $(a,b)$ with $a\in A$ and $b\in B$ is precisely the number of non-root vertices in $A$.  This number is $|A|$ if $r'\in B$ and $|A|-1$ if $r'\in A$.  In either case, this is independent of the choice of spanning tree.



\end{proof}

For the remainder of this section, we will let $A$ and $B$ denote the color classes of $G$ corresponding to regions in the diagram of $L_{G}$ with counterclockwise oriented boundary and clockwise oriented boundary, respectively.

\begin{lemma}
\label{regionlabels}
Suppose each crossing of $L_{G}$ is labeled as in Kauffman's model.  Every meeting between a crossing and a region associated with an element of $A$ is labeled $t$, and every meeting between a crossing and a region associated with an element of $B$ is labeled $1$. (See Figure \ref{fig:crossings}).
\end{lemma}

\begin{proof}
The requirements that every crossing of $L_{G}$ is positive, that the segments of $L_{G}$ surrounding regions in $A$ spin clockwise, and that the segments of $L_{G}$ surrounding regions in $B$ spin counterclockwise imply that each meeting between an edge of $G$ and a crossing of $L_{G}$ has the form shown in Figure \ref{fig:crossings}. The definition of Kauffman's model, thus, yields the desired result.
\end{proof}

\begin{figure}
\centering
\includegraphics[height=5cm]{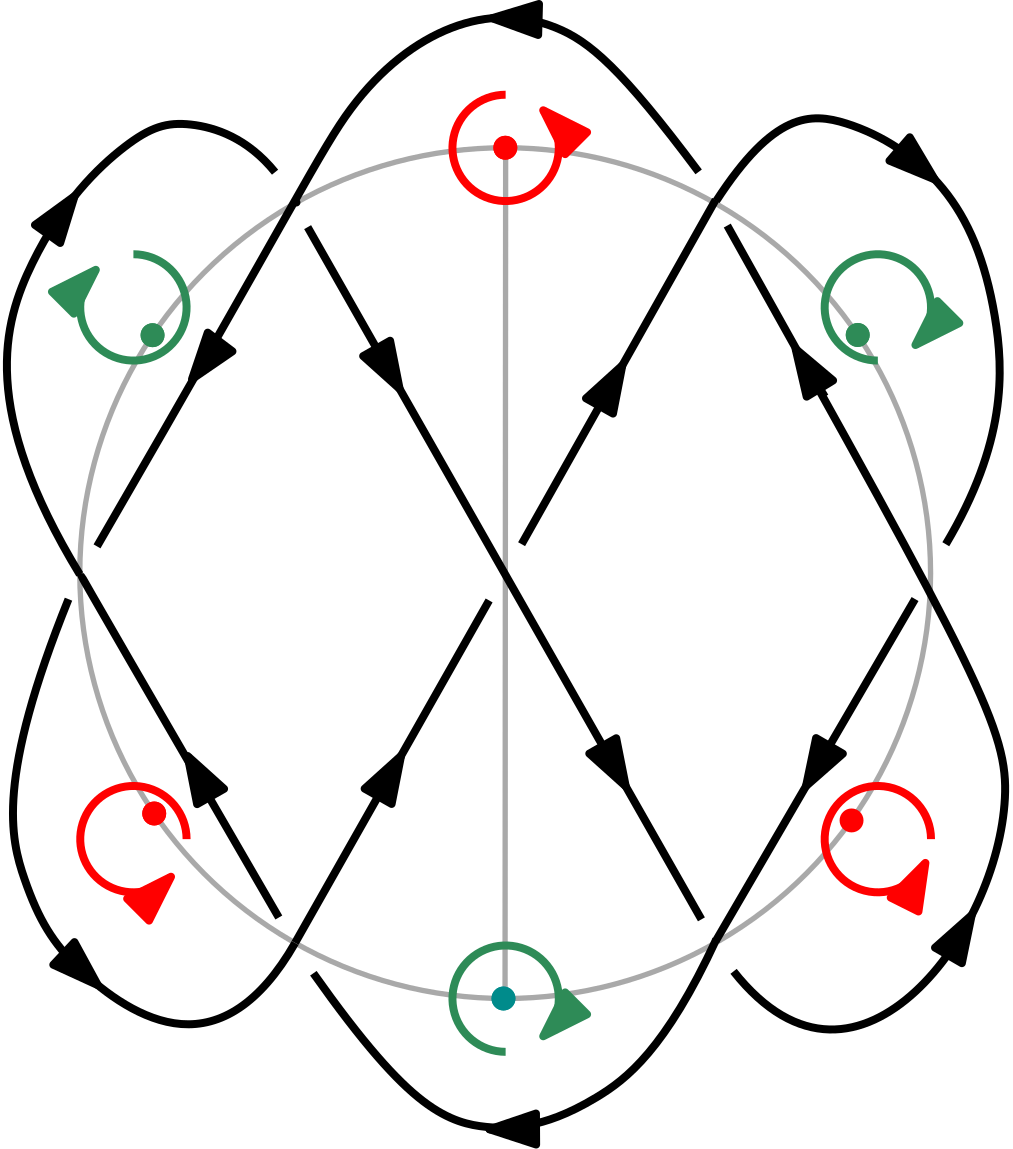}
\caption{An example of a special alternating link with edge orientations around each vertex of the checkerboard graph indicated.}
\label{fig:cw_ccw}
\end{figure}

\begin{figure}
\centering
\includegraphics[height=2.5cm]{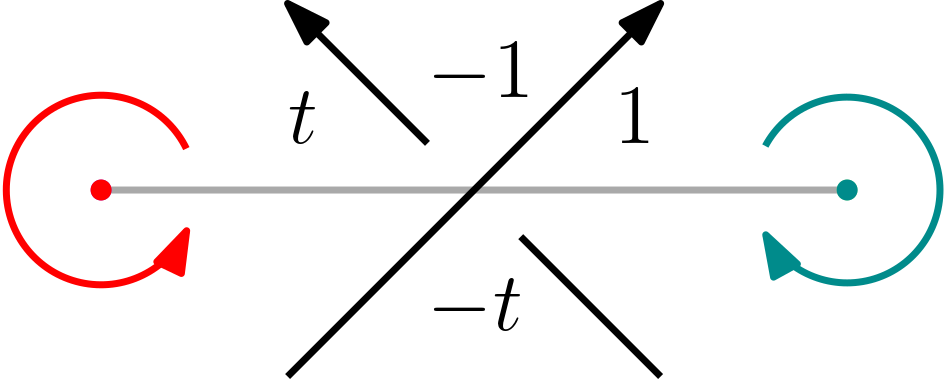}
\caption{The Kauffman weights at each region meeting a crossing with an edge in $G$.}
\label{fig:crossings}
\end{figure}

With this foundation, we proceed with the proof of Theorem \ref{thm2}.
\medskip

\noindent \textit{Proof of Theorem \ref{thm2}.}    
Given an alternating dimap $D$, we construct and orient the link $L_G$ as in Section \ref{sec:SpecialAltDef}.  Supoose $L_G$ is labeled as in Kauffman's model.  By construction, whenever an edge of $D$ meets a crossing in $L_G$, the edge points from a region labeled $-t$ into a region labeled $-1$ (as shown in Figure \ref{fig:example}).

\begin{figure}
\centering
\includegraphics[height=6cm]{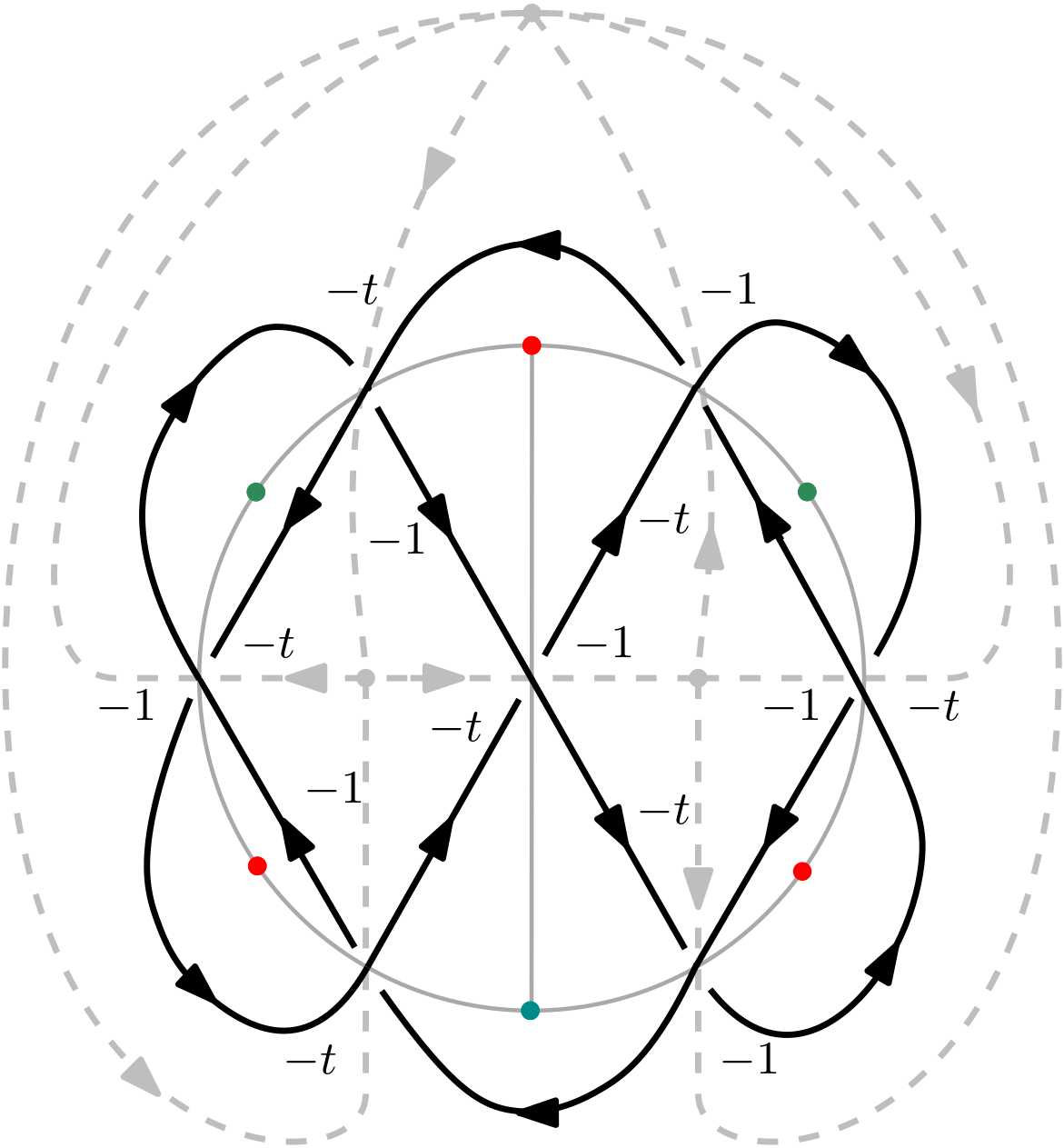}
\caption{A special alternating link $L_G$ with $D$ superimposed.  Here, edges of $D$ are weighted as in Kauffman's model.} 
\label{fig:example}
\end{figure}

Now, let $r'\in\Ver(G)$ and $r\in\Ver(D)$ denote vertices corresponding to adjacent regions in the link diagram $L_G$. Consider dual spanning trees $T$ and $T'$ of $D$ and $G$, respectively. Recall that when forming a state using the dual trees $T'$ and $T$ with roots $r'$ and $r$, an edge $e$ of $T$ corresponds to a state marker with label $-t$ if and only if when $T$ is rooted at $r$, $e$ points from a region labeled $-t$ into a region labeled $-1$ at its corresponding crossing (otherwise, the state marker has label $-1$). By the above construction, this occurs precisely when the orientation of $e$ yielded by rooting $T$ at $r$ matches the orientation of $e$ in $D$.

Let $\Edg_{\text{away}}(T)$ denote the subset of edges of $T$ such that reversing the orientation of each element of $\Edg_{\text{away}}(T)$ makes $T$ an oriented spanning tree. The edges of $\Edg_{\text{away}}(T)$ are precisely those whose orientations after rooting $T$ at $r$ are opposite their orientation in $D$. This implies that $T$ contributes monomial weight $t^{|\Edg(T)\setminus \Edg_{\text{away}}(T)|}$ to $P_D(t)$; in other words, if $T$ is a $k$-spanning tree, it contributes weight $t^{|\Ver(G)|-1-k}$.

Notice that an edge in $T'$ contributes nontrivially to its associated monomial if and only if it is directed \textit{away} from a vertex in $A$. In particular, by Lemma \ref{regionlabels}, the degree of this monomial is the number of edges $(a,b)$ where $a\in A$ and $b \in B$, when $T'$ is rooted at $r$. By Lemma \ref{SpanningTreesLemma}, this degree $m$ is independent of the choice of spanning tree.

All directed edges of $D$  meet crossings of the original link at black holes. Viewing any edge in $D$ as in Figure \ref{fig:crossings}, the directed edge meets the positive crossing from below where the segments of the link are also oriented upwards. The number of black holes in the state defined by $T$ and $T'$ is thus $|\Edg_{\text{away}}(T)|$. As such, no cancellation occurs in the state summation of a special alternating link.

We thus conclude that if $T$ is a $k$-spanning tree and $S$ is the state given by $T$ and $T'$, then $\langle L_G | S\rangle = t^mt^ {|\Ver(G)|-1-k}$, and we can write

\begin{align*}
    \Delta_{L_G}(-t)\sim\sum_S |\langle L_G|S\rangle| & = \sum_{k=0}^{|\Ver(D)|-1} \sum_{S \ \text{such that T is a $k$-spanning tree}} |\langle L_G|S\rangle|\\
    & =\sum_{k=0}^{|\Ver(D)|-1} \sum_{S \ \text{such that T is a $k$-spanning tree}} t^mt^ {|\Ver(D)|-1-k}\\
    & = \sum_{k=0}^{|\Ver(D)|-1}c_k(D,r) t^mt^ {|\Ver(D)|-1-k}\\
    &\sim\sum_{k=0}^{|\Ver(D)|-1}c_k(D,r)t^{|\Ver(D)|-1-k}
\end{align*}
\qed

\begin{remark} \label{rem:K} The proof of Theorem \ref{thm2} shows that in the case of special alternating links $L_G$, the Alexander polynomial $\Delta_{L_G}(-t)=P_D(t)$ can be calculated by collecting terms $t^{|\Edg(T)\setminus \Edg_{\text{away}}(T)|}$ for each spanning tree $T$ of $D$ (with specified arbitrary root $r$). We will later use the notation $\wt_K(T)=t^{|\Edg(T)\setminus \Edg_{\text{away}}(T)|}$ where the subscript $K$ indicates that we are dealing with the Kauffman weight for special alternating links.  
\end{remark}

\section{A note on generalizing Fox's conjecture for special alternating links}
\label{sec:conj}

We note that Conjecture \ref{conj:log-concave} is a generalization of Fox's Conjecture \ref{fox} for special alternating links. Recall Conjecture \ref{conj:log-concave}:

\begin{recallconjecture}
   For any Eulerian digraph $D$, the coefficients of the  polynomial $P_D(t)$ form a log-concave sequence with no internal zeros.
\end{recallconjecture}

While the methods of \cite{hafner2023logconcavity} settled the alternating dimap case of Conjecture \ref{conj:log-concave}, they cannot readily be adapted to all Eulerian digraphs. We offer here another case in which Conjecture \ref{conj:log-concave} holds.  
   
   \medskip
   
The \textbf{Laplacian matrix} of an Eulerian digraph $D$ on the vertex set $[n]=\{1, \ldots, n\}$ is the matrix $L(D)=(l_{i,j})_{i, j \in [n]}$, defined by

$$l_{i,j}=\begin{cases}
\odeg(i) & \text{for} \ i=j \\
|\{e\in \Edg(D) \ : \ $e$ \ \text{is an edge from} \ i \ \text{to} \ j\}| & \text{for} \ i\neq j
\end{cases}.$$

Let $\overline{L(D)}$ be the reduced Laplacian obtained from $L$ by removing a row and a column (of the same index). We will make use of the following straightforward consequence of the matrix-tree theorem.

\begin{lemma}[\cite{even}]
    \label{lem:Laplacian}
    For any Eulerian directed graph, $P_D(t)$ can be expressed via the determinantal formula
$$P_D(t)=\det(\overline{L(D)}+t\overline{L(D)}^T).$$
\end{lemma}

\medskip

   Recall that a sequence of nonnegative integers $a_0, \ldots, a_d$  is ultra log-concave if $$\frac{a_k^2}{{d \choose k}^2}\geq \frac{a_{k+1}}{{d \choose {k+1}}} \frac{a_{k-1}}{{d \choose {k-1}}} \text{ for all }0<k<d.$$  In particular, ultra log-concavity implies log-concavity. 

\begin{proposition} The sequence of coefficients of $P_D(t)$ is ultra log-concave for all Eulerian diagraphs $D$ on the vertex set $[n]$ which have the same number of edges going from vertex $i$ to vertex $j$ as from vertex $j$ to vertex $i$ for all $i,j \in [n]$. 
\end{proposition}

\proof 

If $D$ has the same number of edges going from vertex $i$ to vertex $j$ as from vertex $j$ to vertex $i$ for all $i,j \in [n]$, then $\overline{L(D)}$ is symmetric: $\overline{L(D)} ^T = \overline{L(D)}$. In this case, Lemma \ref{lem:Laplacian} yields: 

$$P_D(t)  = \det(\overline{L(D)}+t\overline{L(D)}^T)=(t+1)^{n-1} \det(\overline{L(D)}).$$

Since the sequence of coefficients of $(t+1)^{n-1}$ is ultra log-concave, we achieve the desired result.

\qed

We note that in general, ultra log-concavity of the coefficients of $P_D(t)$ does not hold. For example, consider the case when $D$ is an oriented three-cycle (see Figure \ref{fig:trefoil_not_ultra}). In this case, the sequence of coefficients of $P_D(t)$ is $(1,1,1)$ which is not ultra log-concave.

\begin{figure}
\centering
\includegraphics[height=5cm]{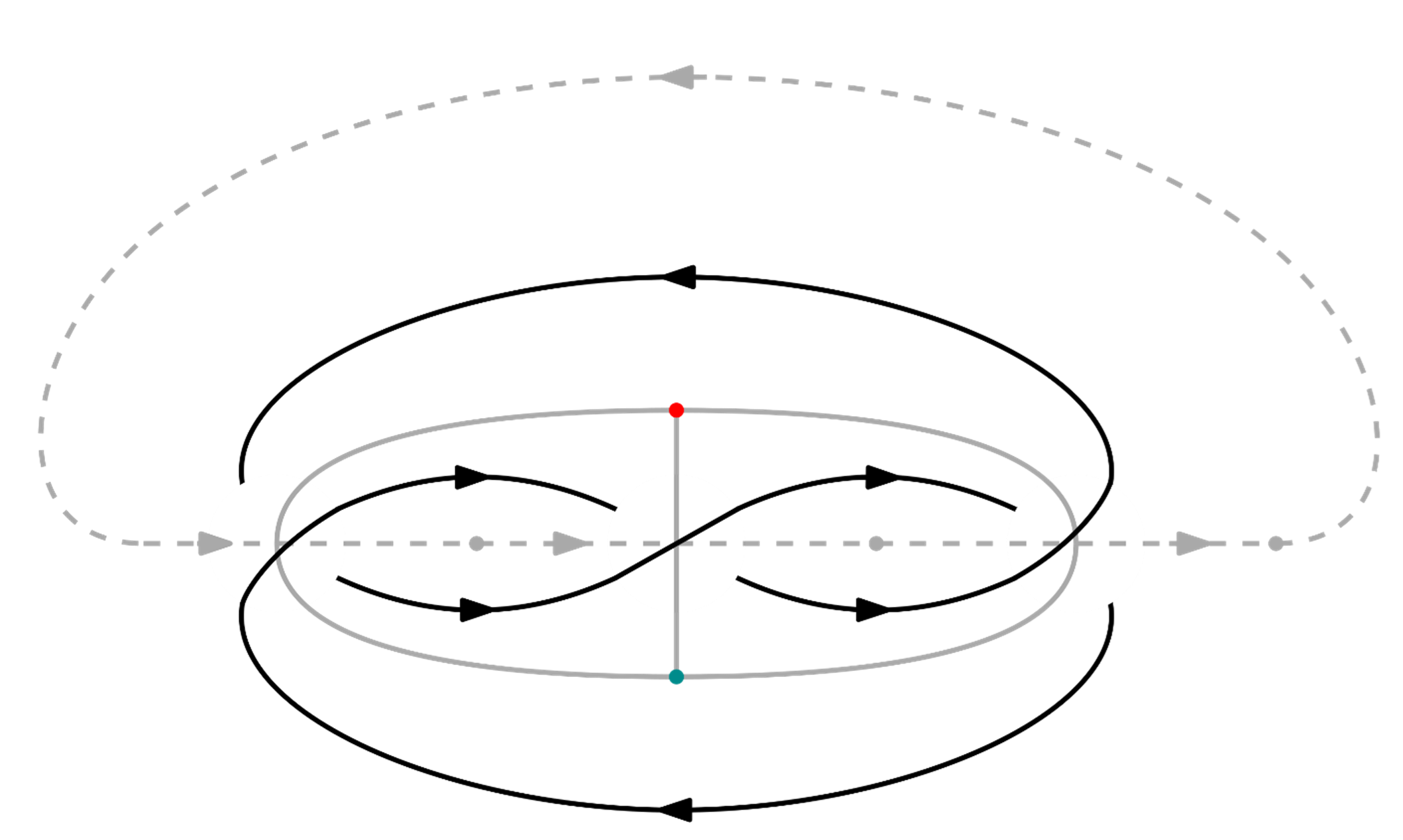}
\caption{The oriented three-cycle $D$, superimposed on a link diagram of the trefoil knot and its checkerboard graph.}
\label{fig:trefoil_not_ultra}
\end{figure}

\section{Bijection between the Crowell and Kauffman state models}
\label{sec:bij}

This section is devoted to relating the first state model discovered for the Alexander polynomial: Crowell's model \cite{Crowell} to the most well-known one: the Kauffman state model \cite{K06}. In Theorem \ref{maintheorem}, we define a weight-preserving bijection between the states of the Crowell model and the states of the Kauffman model for special alternating links.

Throughout this section, $G$ denotes the planar bipartite graph which is the checkerboard graph of the special alternating link $L:=L_G$, and $D$ denotes the directed Eulerian dual of $G$  (see Section \ref{sec:SpecialAltDef} and  Figure \ref{fig:bipandlink}).

 \subsection{Crowell's model} \label{Sec:Crowell} The following combinatorial model for the Alexander polynomial of alternating links is due to Crowell \cite{Crowell}. Recall that a link is \textit{alternating} if it has an \textit{alternating diagram}. A link diagram is \textit{alternating} if its undercrossings and overcrossings alternate as we trace along each component.
 
 Let $\mathcal{G}(L)$ be the planar graph obtained by flattening the crossings of an alternating diagram of $L$: the crossings of $L$ are the vertices of $\mathcal{G}(L)$, and the arcs between the crossings are the edges of $\mathcal{G}(L)$.  Note that $\mathcal{G}(L)$ is a planar $4$-regular graph. Next, we assign directions and weights to the edges of $\mathcal{G}(L)$ in the following way:

\begin{center} \begin{tikzpicture}
\draw[black, thick] (1,0) -- (1,.9);
\draw[black, thick] (1,1.1) -- (1,2);
\draw[-stealth,black, thick] (0,1) -- (2,1);
\draw[-stealth,black,thick] (5,0)--(5,1);
\draw[-stealth,black,thick] (5,2)--(5,1);
\draw[black,thick] (4,1) --(6,1);
\node at (3,1.2) {becomes};
\node at (4.8,1.5) {$1$};
\node at (4.7,.5) {$-t$};
\end{tikzpicture} \end{center} 

On the left, we see the orientation of the link $L$ at a crossing, and on the right, we see how the edges of $\mathcal{G}(L)$ incident to the corresponding vertex are directed and weighted. Denote by $\overrightarrow{\mathcal{G}(L)}$ the oriented weighted graph obtained from $\mathcal{G}(L)$ in this fashion. Note that the edge orientations assigned to $\overrightarrow{\mathcal{G} (L)}$ are not the same as the orientation of the corresponding components of the link.  Let $\var(e)$ denote the weight $-t$ or $1$ assigned to the edge $e \in \Edg(\overrightarrow{\mathcal{G}(L)})$. See Figure \ref{fig:bipandlink1} for a full example.

Recall that an \text{arborescence} of $\overrightarrow{\mathcal{G}(L)}$ rooted at $v\in\Ver(\overrightarrow{\mathcal{G}(L)})$ is a directed subgraph $T$ such that for each other vertex $u$, there is a unique directed path in $T$ from $v$ to $u$. We will denote by  $\mathcal{A}_v(\overrightarrow{\mathcal{G}(L)})$ the set of arborescences of $\overrightarrow{\mathcal{G}(L)}$ rooted at $v$.  

\begin{notation}\label{not:crow_weight}
    
For every $A \in \mathcal{A}_v(\overrightarrow{\mathcal{G}(L)})$, we define the \textbf{Crowell weight} to be
\[\wt_C(A)=\prod_{e\in E (A)} \var(e).\]

\end{notation}

\begin{figure}
\centering
\includegraphics[height=5.5cm]{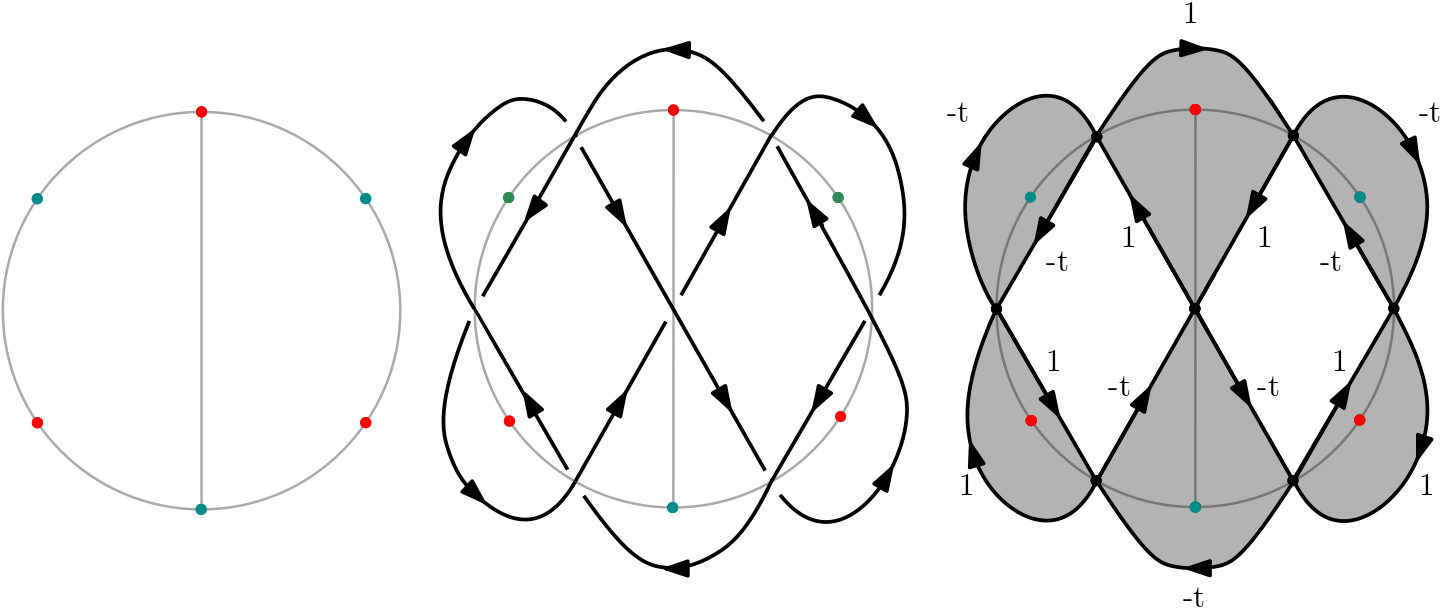}
\caption{A planar bipartite graph $G$ along with  its associated positive special alternating link $L$ and Crowell's  edge weighted digraph $\protect\overrightarrow{\mathcal{G}(L)}.$}
\label{fig:bipandlink1}
\end{figure}

\begin{theorem}[\cite{Crowell} Theorem 2.12]
    \label{Thm:CrowellModel} Given an alternating diagram of the link $L$, fix an arbitrary vertex $v \in \Ver\overrightarrow{\mathcal{G}(L)})$. The Alexander polynomial of $L$ is: 
  
    \[\Delta_L(t)\sim\sum_{\substack{A\in \mathcal{A}_v(\overrightarrow{\mathcal{G}(L)})}} \wt_C(A). \]
\end{theorem}


\subsection{Definition of our weight-preserving bijection}
The goal of this section is to define a bijection between arborescences of $\overrightarrow{\mathcal{G}(L)}$ and spanning trees of $D$ that is weight-preserving up to multiplication by a fixed monomial. For fixed $v\in\Ver(\overrightarrow{\mathcal{G}(L)})$ and $A\in\mathcal{A}_v(\overrightarrow{\mathcal{G}(L)})$, recall that the \textit{Crowell weight} of an arborescence $A$ is the monomial $\wt_C(A)=\prod_{e\in\Edg(A)} \var(e)$. We define the \textit{Kauffman weight} of a spanning tree of $D$ rooted at $r\in\Ver(D)$ as follows.

\begin{notation}\label{not:kauff_weight}

For fixed $r\in\Ver(D)$ and for $T\in\mathcal{T}(D)$, we will write $\wt_K(T)=t^{|\Edg(T)\setminus \Edg_{\text{away}}(T)|}$ for the \textbf{Kauffman weight}, as explained in Remark \ref{rem:K}, where $\Edg_{\text{away}}(T)$ is the subset of edges of $T$ such that reversing the orientation of every edge in $\Edg_{\text{away}}(T)$ would make $T$ an oriented spanning tree rooted at $r$. 


\end{notation}

Fix $v\in\Ver(\overrightarrow{\mathcal{G}(L)})$ in the boundary of the exterior region of $\overrightarrow{\mathcal{G}(L)}$, and let $f_v\in\Edg(\overrightarrow{\mathcal{G}(L)})$ be the unique edge in the boundary of the exterior region with final vertex $v$. Let $\mathcal{C}(\overrightarrow{\mathcal{G}(L)})$ denote the set of Eulerian tours of $\overrightarrow{\mathcal{G}(L)}$ considered as paths beginning with $f_v$.  Let $\mathcal{T}(D)$ denote the set of spanning trees of $D$. Recall that $\mathcal{G}(L)$ is the medial graph of $G$ (see Section \ref{sec:SpecialAltDef}).  Since $D$ (considered without its orientation) is the planar dual of $G$, $\overrightarrow{\mathcal{G}(L)}$ can equivalently be constructed by assigning edge weights and orientations to the medial graph of $D$.  In this section, we will write $v_e$ to mean the vertex of $\overrightarrow{\mathcal{G}(L)}$ corresponding to an edge $e\in \Edg(D)$, as in this medial graph construction. 

\begin{definition} \label{def:bitransition}
    Let $\Phi:\mathcal{T}(D)\rightarrow \mathcal{C}(\overrightarrow{\mathcal{G}(L)})$ be the map which, given $T\in \mathcal{T}(D)$, produces $\Phi(T)$ by the following procedure. At each vertex $v_e$ of $\overrightarrow{\mathcal{G}(L)}$, we choose a pairing of incident edges as shown below, depending on whether or not $e$ is in $\Edg(T)$.
    \vspace{.5cm}
    
    \includegraphics[height=3.5cm]{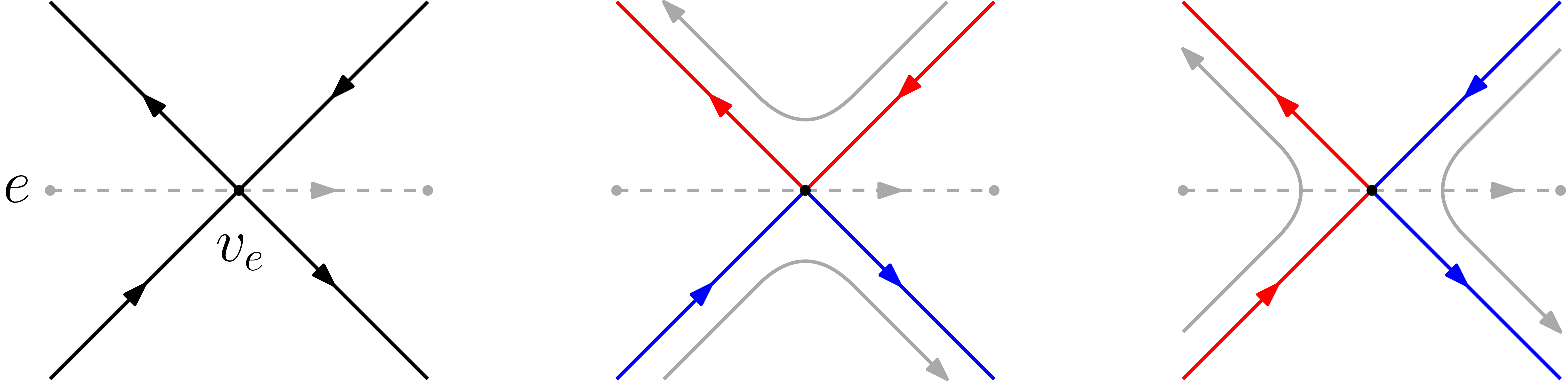}
    
    
    \vspace{.5cm}
    On the left, we show an edge $e \in \Edg(D)$ meeting the vertex $v_e$ in $\overrightarrow{\mathcal{G}(L)}$.  The center diagram shows the choice of edge pairing in $\overrightarrow{\mathcal{G}(L)}$ when $e\in \Edg(T)$, and the right diagram shows the choice of edge pairing in $\overrightarrow{\mathcal{G}(L)}$ when $e\notin \Edg(T)$.
\end{definition}

Figure \ref{fig:Phi_map} depicts an example of a spanning tree and its image under $\Phi$.  

\begin{definition} \label{def2}
    Let $\Psi_{v, f_v}:C(\overrightarrow{\mathcal{G}(L)})\rightarrow \mathcal{A}_v(\overrightarrow{\mathcal{G}(L)})$ be the following map. Given an Eulerian tour $C$ of $\overrightarrow{\mathcal{G}(L)}$, $\Psi_{v,f_v}(C)$ is the subgraph of $\overrightarrow{\mathcal{G}(L)}$ formed by tracing $C$, starting at the edge $f_v$, and at each $u\in\Ver(\overrightarrow{\mathcal{G}(L)})\setminus{v}$, selecting the first edge incoming to $u$.
\end{definition}

An eloquent exposition of the map from Definition \ref{def2} can be found in Stanley's account of the BEST Theorem \cite{stanley2013algebraic}. Recall that  $\overrightarrow{\mathcal{G}(L)}^T$ denotes the \text{transpose} of $\overrightarrow{\mathcal{G}(L)}$, i.e. the digraph obtained by reversing the direction of each edge in $\overrightarrow{\mathcal{G}(L)}$. Applying \cite[Theorem 10.2]{stanley2013algebraic}  to $\overrightarrow{\mathcal{G}(L)}^T$ with the additional condition that every vertex of $\overrightarrow{\mathcal{G}(L)}$ has outdegree $2$, we readily obtain that  $\Psi_{v, f_v}$ is a bijection. Refer back to Section \ref{sec:k} for more details on the BEST theorem. 

The map  $F=\Psi_{v,f_v}\circ\Phi$ is a bijection between spanning trees of $D$ and arborescences of $\overrightarrow{\mathcal{G}(L)}$. In fact, Theorem \ref{maintheorem} says a lot more: $F$ is weight-preserving (up to multiplication by a fixed monomial which depends only on graph-theoretic properties of $D$) when the spanning trees of $D$ and arborescences of $\overrightarrow{\mathcal{G}(L)}$ are weighted as in Notation \ref{not:kauff_weight} and Notation \ref{not:crow_weight}, respectively. 

We set the following convention to make sense of the notation $m_2$ used in the statement of Theorem \ref{maintheorem}.  Refer to the clockwise oriented regions of $\overrightarrow{\mathcal{G}(L)}$ by $R(\overrightarrow{\mathcal{G}(L)})$. It was shown in \cite[Lemma 3.1]{hafner2023logconcavity} that for each of these regions, the edges are either labeled exclusively by $-t$ or exclusively by $1$  in Crowell's model. Moreover, for any pair of adjacent vertices in $\Ver(G)$, one of the corresponding regions must be labeled with $-t$ and the other with $1$.  We can thus sort the elements of $R(\overrightarrow{\mathcal{G}(L)})$ into color classes $\{R_1,\dots,R_{m_1}\}\sqcup\{R_{\overline{1}},\dots, R_{\overline{m_2}}\}$ where $\{R_1,\dots,R_{m_1}\}$ are the regions whose boundary edges are labeled by $-t$ in Crowell's model.  We refer to $\{R_1,\dots,R_{m_1}\}$ and $\{R_{\overline{1}},\dots, R_{\overline{m_2}}\}$ as \textbf{gold} and \textbf{blue} regions respectively. 

\begin{remark}
    Let $G$ be a planar bipartite graph. Comparing the definitions of gold and blue with Figure \ref{fig:crossings}, a region in $R(\overrightarrow{\mathcal{G}(L_G)})$ is gold if and only if the corresponding region of $L_G$ has clockwise oriented boundary. 
\end{remark}

\begin{figure}
\centering
\includegraphics[width=13cm]{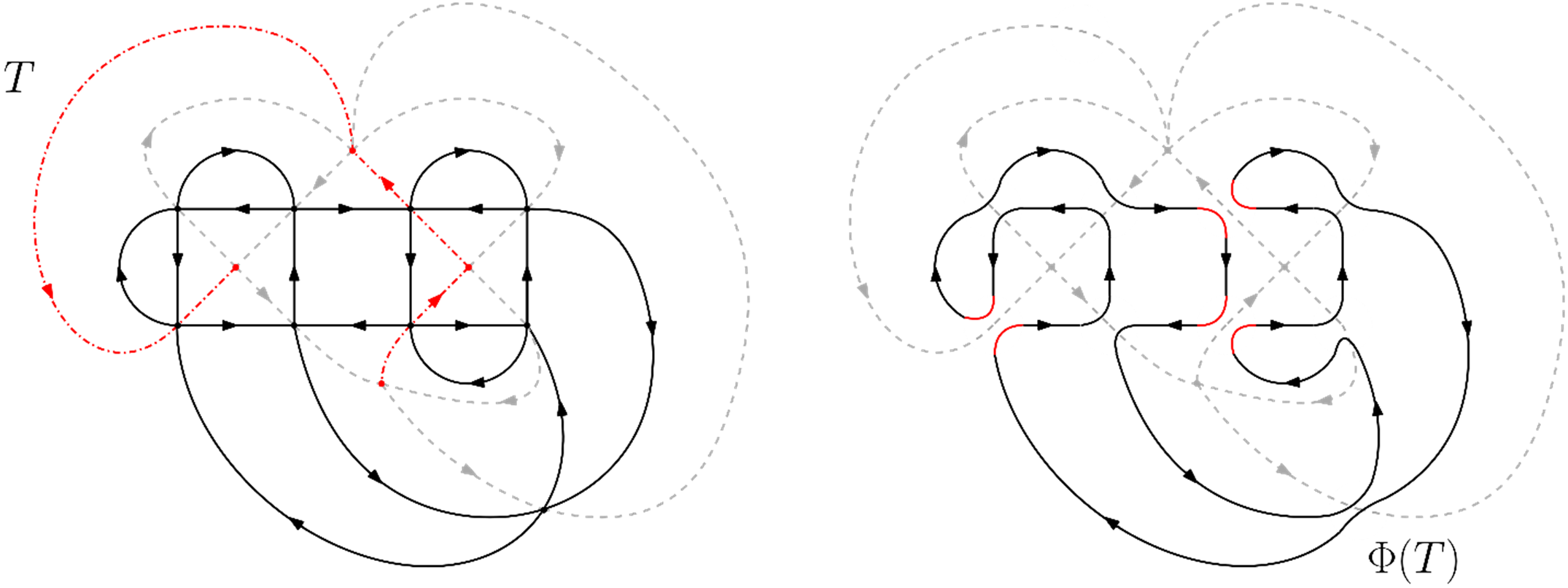}

\caption{In the above image, $D$ is depicted in dashed grey, and $\protect\overrightarrow{\mathcal{G}(L)}$ is depicted in black. On the left is a spanning tree $T$ of $D$. On the right is its image under $\Phi$.}
\label{fig:Phi_map}
\end{figure}

\begin{theorem}
\label{maintheorem}
Let $r\in\Ver(D)$ denote the vertex corresponding to the exterior region of $L$. Fix $e_0 \in \Edg(D)$ with $\fin(e_0)=r$. Let $v=v_{e_0} \in \Ver(\overrightarrow{\mathcal{G}(L)})$, and let $f_v\in\Edg(\overrightarrow{\mathcal{G}(L)})$ be the unique edge in the boundary of the exterior region such that $\fin (f_v) = v$.

Then $F=\Psi_{v,f_v}\circ\Phi$ is a bijection between spanning trees of $D$ and arborescences of $\overrightarrow{\mathcal{G}(L)}$. Moreover, ${\rm wt}_C(A)={\rm wt}_K(T)t^{m_2-1}$ where $T \in \mathcal{T}(D)$, $A  = F(T) \in  \mathcal{A}_v(\overrightarrow{\mathcal{G}(L)})$, and $m_2$ is the number of blue regions, as defined above. \end{theorem}

Figure \ref{fig:f_map} depicts an instance of the bijection $F$. 

\begin{figure}
\centering
\includegraphics[width=13cm]{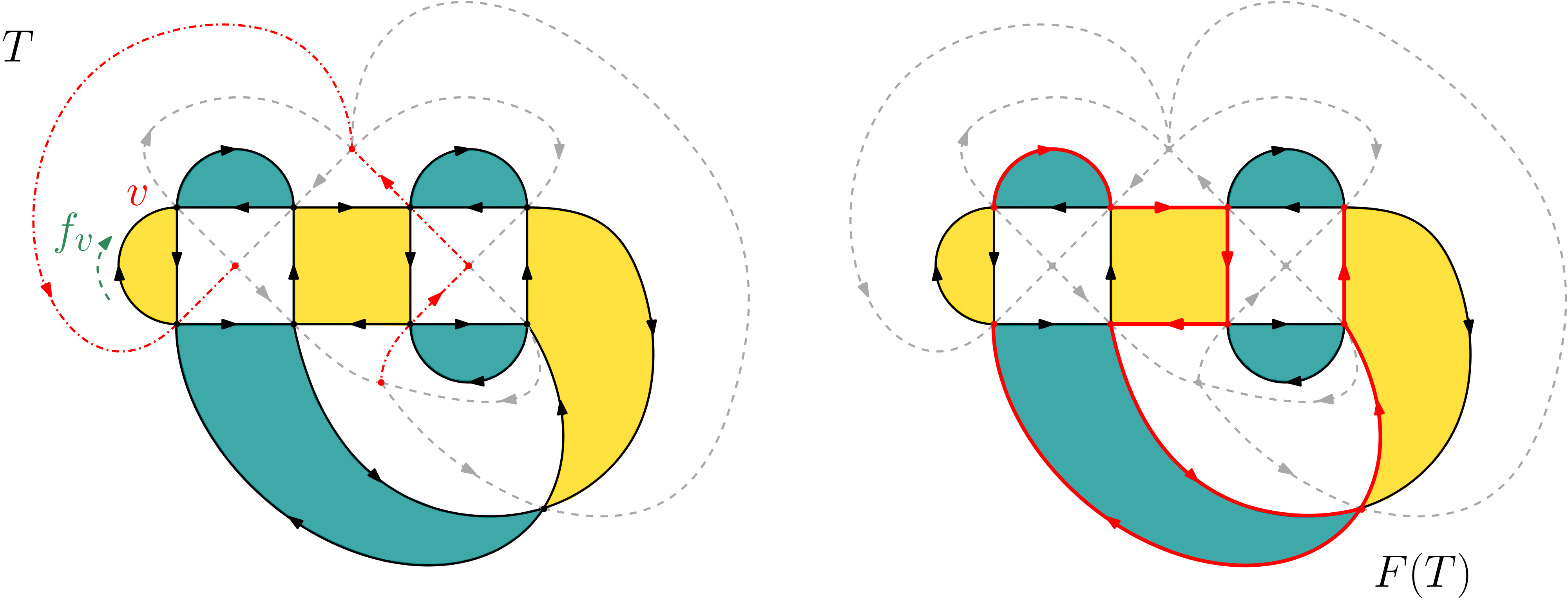}
\caption{A spanning tree of $D$ and its image under $F$.} 
\label{fig:f_map}
\end{figure}

\subsection{Proof of Theorem \ref{maintheorem}}
To prove Theorem \ref{maintheorem}, we introduce four lemmas and accompanying notation.

\begin{figure}
\centering
\includegraphics[width=13cm]{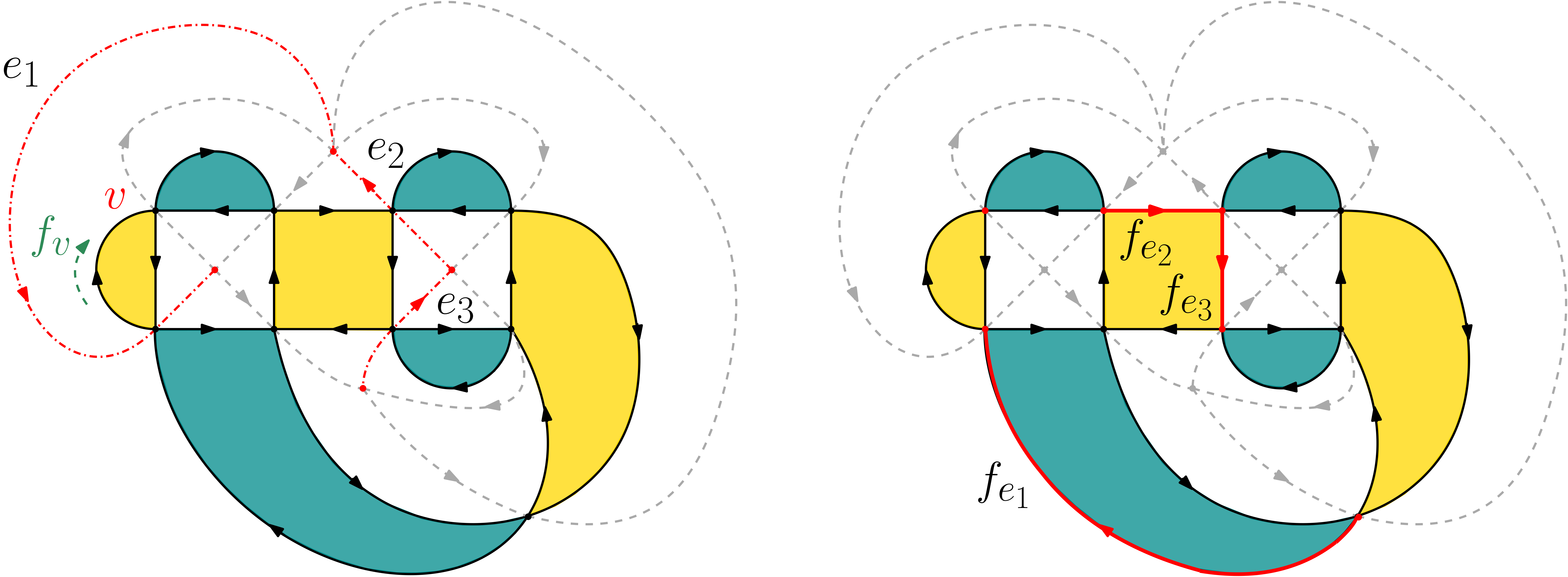}
\caption{A spanning tree of $D$ and for each $e_i\in\Edg(T)$, the associated edge $f_{e_i}\in\Edg(F(T))$.} 
\label{fig:tree_and_edges}
\end{figure}

\begin{definition} \label{def:fe}
    Fix $T\in \mathcal{T}(D)$. For $e \in \Edg(T)$ such that $v_e \neq v$, the root of $F(T)$, let $f_e\in\Edg(F(T))$ denote the unique edge of the arborescence $F(T)$ with final vertex $v_e$ (see Figure \ref{fig:tree_and_edges}).  In particular, $f_e$ is the first edge incoming to $v_e$ in $\Phi(T)$.
\end{definition}

\begin{lemma}
    \label{lem:whathappenstothepaths}
    Fix $T \in \mathcal{T}(D)$, let $w\in\Ver(D)$, and let $e_1 \in \Edg(T)$ be incident to $w$.  Let $f=f_{e_1}$ in $\Phi(T)$ and let $g$ be the other edge of $\overrightarrow{\mathcal{G}(L)}$ with final vertex $v_{e_1}$.  If $f$ is in the boundary of the region of $\overrightarrow{\mathcal{G}(L)}$ corresponding to $w$, then for any $e_2 \in \Edg(T)$ with vertex $w$ such that $e_2 \neq e_1$, the portion of $\Phi(T)$ between the edges $f$ and $g$ does not meet $v_{e_2}$.
\end{lemma}
\begin{proof}

    Observe that if $e',e'' \in \Edg(T)$ are such that $\Phi(T)$ reaches $v_{e'}$ before $v_{e''}$ and no other vertex in the path between $v_{e'}$ and $v_{e''}$ in $\Phi(T)$ has the form $v_{e'''}$ for any $e''' \in \Edg(T)$, then $e'$ and $e''$ share an endpoint in $D$.  In particular, $e'$ and $e''$ are either adjacent or equal.
    
    Let $v_{e_1}, v_{e'_1}\ldots, v_{e'_m}, v_{e_1}$ be the set of all vertices of the form $v_{e'}$ for $e' \in \Edg(T)$ which are reached by $\Phi(T)$ in between the edges $f$ and $g$. Suppose (in search of a contradiction) that $v_{e'_k}$ has $w$ as an endpoint for some $1\leq k \leq m$.  By the observation above, we can see that $e_1, e'_1, \ldots, e'_k$ is a path from $w$ back to $w$ within $T$, and since $e_1 \neq e'_k$, this path must contain a cycle.  This contradicts the fact that $T$ is a tree.  We thus conclude that the lemma holds. 
\end{proof}

Given a spanning tree $T$ of $D$, understanding the weight of $F(T)$ requires one to know precisely how many edges of $F(T)$ border blue regions of $\overrightarrow{\mathcal{G}}(L)$. Lemma \ref{lem:FrogJumping} concerns the edges of the form $f_e$ for some $e\in \Edg(T)$, and Lemma \ref{lem:violetsnew} concerns the remaining edges. 

\begin{lemma}
\label{lem:FrogJumping}
  Fix a spanning tree $T$ of $D$, and let $r\in\Ver(D)$ denote the vertex corresponding to the exterior region of $\overrightarrow{\mathcal{G}(L)}$.  Fix $e_0 \in \Edg(D)$ with $\fin(e_0)=r$. Let $v=v_{e_0} \in \Ver(\overrightarrow{\mathcal{G}(L)})$, and let $f_v\in\Edg(\overrightarrow{\mathcal{G}(L)})$ be the unique edge in the boundary of the exterior region such that $\fin (f_v) = v$. 
  

  For $e\in\Edg(T)$, $f_e$ is blue if and only if $e\in\Edg_{\text{away}}(T)$. 
  

\end{lemma}

\begin{proof}
    The lemma statement is equivalent to saying the edges in $\{f_e \ | \ e\in\Edg(T) \ \text{points towards} \ r  \text{ in $T$} \}$ all border gold regions of $\overrightarrow{\mathcal{G}(L)}$, and the edges in $\{f_e \ | \ e\in\Edg(T) \ \text{points away from} \ r \text{ in $T$}  \}$ all border blue regions of $\overrightarrow{\mathcal{G}(L)}$. For ease of reference, we say an edge of $\overrightarrow{\mathcal{G}(L)}$ is \textit{gold} (resp. \textit{blue}) if it borders a gold (resp. blue) region. 
    
    Fix a spanning tree $T$ of $D$.  Recall that for every $e\in\Edg(T)$, Definitions \ref{def2} and \ref{def:fe} state that $f_e$ is the first incoming edge to $v_e$ in $\Phi(T)$, where we treat $\Phi(T)$ as beginning with the edge $f_v$.  We outline the following procedure to determine the color of $f_e$ for each $e\in\Edg(T)$. 

    \begin{itemize} 
    \item[]Case 1: The edge $e$ of $T$ has $r$ as an endpoint.
    
    By definition, this is equivalent to the case that $v_e$ is in the boundary of the exterior region of $\overrightarrow{\mathcal{G}(L)}$.  Let $P$ denote the portion of the oriented boundary of the exterior region of $\overrightarrow{\mathcal{G}(L)}$ from $v$ to $v_e$ (see the left and center images of Figure \ref{fig:frogjump}). We claim that the last edge of $P$ is the first edge incoming to $v_e$ along $\Phi(T)$ (which is $f_e$). Note that $\Phi(T)$, starting at $f_v$,  can only leave the boundary of the exterior region when it meets a vertex of the form $v_{e'}$ for some $e' \in \Edg(T)$.  Note also that upon leaving the boundary of the exterior region from a vertex of the form $v_{e'}$ for some $e' \in \Edg(T)$, Lemma \ref{lem:whathappenstothepaths} states that the Eulerian cycle $\Phi(T)$ will return to $v_{e'}$ before meeting any other vertex $v_{e''}$, $e'' \in \Edg(T)$,  in the boundary of the exterior region.
    
    Accordingly, the first edge incident to $v_e$ along $\Phi(T)$ is blue if and only if the last edge of $P$ is blue. This happens if and only if $e$ is directed away from $r$ since the edges incident to $r$ alternate between incoming and outgoing. The left and center images in Figure \ref{fig:frogjump} show examples of this case.

    \begin{figure}
    \centering
    \includegraphics[width=16cm]{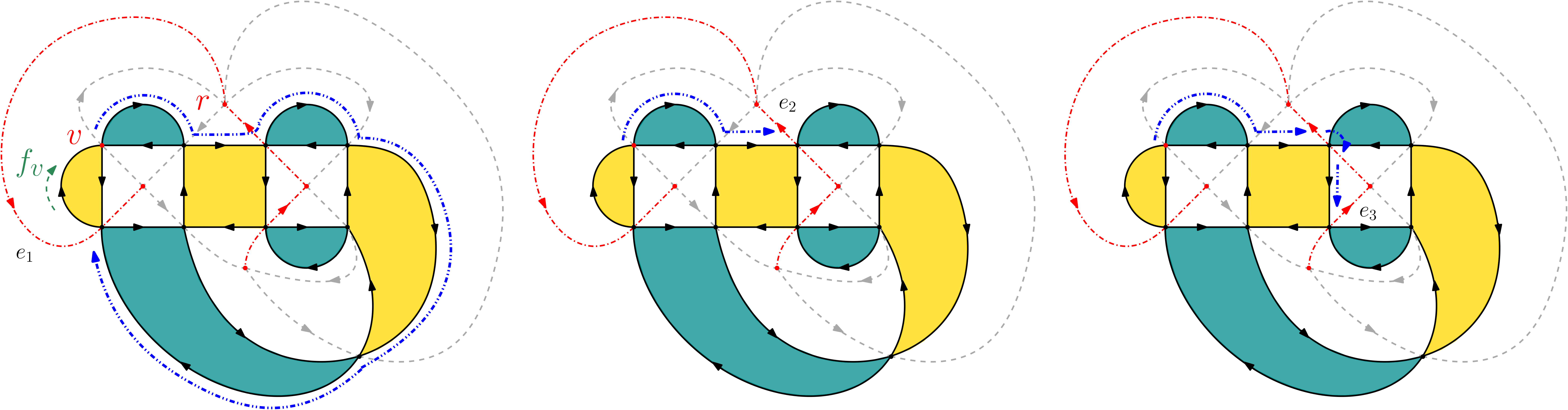}
    \caption{The procedure described in the proof of Lemma \protect\ref{lem:FrogJumping} to determine the first edge incident to $v_{e_i}$ for each of the edges of the spanning tree from Figures \protect\ref{fig:f_map} and \protect\ref{fig:tree_and_edges}.  In the left and center images, the path $P$ is indicated in blue.  In the right image, the blue line indicates the path $P'$ from $v_{e_2}$ to $v_{e_3}$ appended to the path $P$ from $v$ to $v_{e_2}$.} 
    \label{fig:frogjump}
    \end{figure}

  \item[] Case 2: The edge $e\in\Edg(T)$ is not incident to $r$. 
  
  Fix the shortest (undirected) path in $T$ from $r$ to a vertex of $e$, and let $e'\neq e$ denote the final edge of this path. The edges $e$ and $e'$ meet at a single vertex $w$. Let $P'$ denote the oriented path from $v_{e'}$ to $v_e$ in the boundary of the region of $\overrightarrow{\mathcal{G}(L)}$ corresponding to $w$ (in the rightmost image of Figure \ref{fig:frogjump}, $e$ is $e_3$, $e'$ is $e_2$, and $w$ is the vertex where they meet). As in Case 1, Lemma \ref{lem:whathappenstothepaths} implies that the last edge of $P'$ is the first edge incident to  $v_e$ along $\Phi(T)$. If the length of the path connecting $v_{e'}$ and $v_e$ is odd, then since the colors of the edges surrounding any region of $\overrightarrow{\mathcal{G}(L)}$ bounded by a counterclockwise cycle alternate and $D$ is an alternating dimap, the colors of the first edges incident to $v_{e'}$ and $v_e$ are identical, and $e'$ and $e$ either both point towards $v$ or both point away from $v$. Similarly, if the length of the path connecting $v_{e'}$ and $v_e$ is even, then the colors of the first edges incident to $v_{e'}$ and $v_e$ are different, and exactly one of these edges points towards $v$.  We thus achieve the desired result. 
  \end{itemize}
\end{proof}


\begin{lemma}
\label{lem:violetsnew}
    Fix a spanning tree $T$ of $D$, and let $r\in\Ver(D)$ denote the vertex corresponding to the exterior region of $L$. Fix $e_0 \in \Edg(D)$ with $\fin(e_0)=r$. Let $v=v_{e_0} \in \Ver(\overrightarrow{\mathcal{G}(L)})$, and let $f_v\in\Edg(\overrightarrow{\mathcal{G}(L)})$ be the unique edge in the boundary of the exterior region such that $\fin (f_v) = v$. 

    Let $f$ be a blue edge of $\overrightarrow{\mathcal{G}(L)}$, and let $R$ denote the blue planar region it bounds. Suppose $\fin (f)\neq v_e$ for any $e\in\Edg(T)$. Then $f$ is an edge of $F(T)$ if and only if $f$ is not the last edge in the boundary of $R$ to appear in $\Phi(T)$.  
\end{lemma}

\begin{proof}
First, suppose that $f$ is the last edge in the boundary of $R$ to appear in $\Phi(T)$.  Then $f$ is the first edge pointing into $\fin(f)$ in $\Phi(T)$ only when $\fin(f)=v$, so $f \notin \Edg(F(T))$.

Next, suppose $f$ is not the last edge in the boundary of $R$ reached by $\Phi(T)$. Let $f'$ be the unique edge such that $\init(f')=\fin(f)$. There are two possibilities. 
\begin{enumerate}
    \item Either the edge $f'$ appears before $f$ in $\Phi(T)$, or
    \item the edge $f'$ appears after $f$ in $\Phi(T)$.
\end{enumerate}
We argue that in the first case, $f\in\Edg(F(T))$ and that the second case is impossible.

    \begin{enumerate}

        \item []Case 1: The final vertex $\fin(f)=v_e$ for some edge $e \in \Edg(D)$ which is not in $T$, and the unique edge $f'$ in the boundary of $R$ with $\fin(f)=\init(f')$ appears after $f$ in the Eulerian tour $\Phi(T)$.

        In this case, the only other edge of $\overrightarrow{\mathcal{G}(L)}$ with final vertex $\fin(f)$ must appear immediately before $f'$ in $\Phi(T)$; therefore, $f$ is the first edge pointing into $\fin(f)$ in $\Phi(T)$.  By the definition of $\Psi_{v,f_v}$, this implies $f$ is an edge of $F(T)$.  
        
        \item []Case 2: The final vertex $\fin(f)=v_e$ for some edge $e \in \Edg(D)$ which is not in $T$, and the unique edge $f'$ in the boundary of $R$ with $\fin(f)=\init(f')$ appears before $f$ in the Eulerian tour $\Phi(T)$. 

       Assuming Case 2, note that there must be a path $P$ in $\Phi(T)$ from $f'$ to $f$, as shown below.
        
        \begin{center} \begin{tikzpicture}
        \filldraw[darkcyan] (1,1) circle (1);
        \draw[gray, thick] (1,1) circle (1);
            \draw[-stealth,black, thick] (.25,1.9) -- (0,1.1) ;
            \draw[black,thick] (0,1.1) .. controls (-.25,1.05) .. (-.5,1.15);
            \draw[-stealth,black, thick] (0,.9) -- (.25,.1) ;
              \draw[black,thick] (0,.9) .. controls (-.25,.95) .. (-.5,.85);
            \draw[black,thick, dashed] (.25,1.9) .. controls (1.9,2.5) .. (2.5,1) .. controls (1.9,-.5) .. (.25,.1); 
           \node at (-.1,1.5) {$f$};
           \node at (-.1,.5) {$f'$};
           \node at (1,1) {$R$};
           \node at (2.6,1.5) {P};
        \end{tikzpicture} \end{center}

Suppose that $g\in \Edg(\overrightarrow{\mathcal{G}(L)})$ is an edge not in $P$ such that $\fin(g)\in\Ver(P)$.  Let $g'$ be the edge immediately following $g$ in $\Phi(T)$; in particular, $\init(g')=\fin(g)$.  By construction of $\Phi$, $g$ and $g'$ must either both lie inside the region bounded by $P$ or both lie outside it (see Definition \ref{def:bitransition}).  This implies that once $\Phi(T)$ reaches $f$, it cannot re-enter the portion of the graph enclosed by $P$.  In particular, no edge in the boundary of $R$ can be reached after $f$ in $\Phi(T)$. This contradicts the assumption that $f$ is not the last edge bounding $R$ that is reached by $\Phi(T)$. 
    \end{enumerate}    
\end{proof}

\begin{lemma}
    \label{lem:overlap}
    Fix a spanning tree $T$ of $D$, and let $r\in\Ver(D)$ denote the vertex corresponding to the exterior region of $L$. Fix $e_0 \in \Edg(D)$ with $\fin(e_0)=r$. Let $v=v_{e_0} \in \Ver(\overrightarrow{\mathcal{G}(L)})$, and let $f_v\in\Edg(\overrightarrow{\mathcal{G}(L)})$ be the unique edge in the boundary of the exterior region such that $\fin (f_v) = v$. 

    Let $f \in \Edg (\overrightarrow{\mathcal{G}(L)})$ be in the boundary of a blue region $R$, and suppose that $f$ is the last edge of $R$ to be reached by $\Phi(T)$.  Then $\fin(f) \neq v_e$ for any $e \in T$.
\end{lemma}
\begin{proof}
    Suppose $\fin(f) = v_e$ for some $e \in T$.  Let $f'$ be the unique edge on the boundary of $R$ such that $\init (f') = v_e$.  Note that $f'\neq f_v$.  Otherwise, we would have $f'=f_{e_0}$ where $\fin(e_0)=r$.  Lemma \ref{lem:FrogJumping} would then imply that $f'$ is gold, yielding a contradiction.
    
    Since $f'\neq f_v$, Definition \ref{def:bitransition} requires that $f'$ immediately follows $f$ in $\Phi(T)$.  This contradicts our assumption that $f$ is the last edge of $R$ to be reached by $\Phi(T)$.
\end{proof}

\noindent \textit{Proof of Theorem \ref{maintheorem}.}   Let $w$ denote the degree of  ${\rm wt}_C(F(T))$. For a region $R\in R(\overrightarrow{\mathcal{G}(L)})$, let $\partial R$ denote the oriented boundary of $R$ as a subgraph of $\overrightarrow{\mathcal{G}(L)}$. Let $\mathcal{B}$ denote the set of blue edges of $\overrightarrow{\mathcal{G}(L)}$.

\begin{align*}
   w &= \left|\{f\in\Edg(F(T)) \ | \ f \ \text{is} \ \text{gold}\}\right| \\
   &= |\Edg(F(T))| - \left|\{f\in\Edg(F(T)) \ | \ f \ \text{is} \ \text{blue}\}\right|\\
   &= (|\Ver(\overrightarrow{\mathcal{G}(L)})|-1) - \left|\{f\in\Edg(F(T)) \ | \ f \ \text{is} \ \text{blue}\}\right|\\
   &= (|\Ver(\overrightarrow{\mathcal{G}(L)})|-1) - \left|\{f\in\Edg(F(T)) \ | \ f=f_e \ \text{for some} \ e\in T \ \text{and} \ f \ \text{is} \ \text{blue}\}\right| \\
   &\phantom{={}} - \left|\{f\in\Edg(F(T)) \ | \ f\neq f_e \ \text{for any} \ e\in T \ \text{and} \ f \ \text{is} \ \text{blue}\}\right|\\
 \end{align*}

 We consider the four possible cases for each blue edge $f\in\Edg(\overrightarrow{\mathcal{G}}(L))$.

 \begin{enumerate}
     \item []Case 1: The final vertex $\fin(f)=v_e$ for some $e\in \Edg_{\text{away}}(T)$.  In this case, we must have $f=f_e$.  Otherwise, $f_e$ would be a gold edge (since there are only two edges of $\overrightarrow{\mathcal{G}(L)}$ with a given final vertex -- one blue and one gold), and Lemma \ref{lem:FrogJumping} would imply that $e$ is directed towards $r$, yielding a contradiction. 
    \item []Case 2: The final vertex $\fin(f)=v_e$ for some $e\in \Edg(T)\setminus \Edg_{\text{away}}(T)$. Then, $f$ cannot be an edge of $F(T)$, since by Lemma \ref{lem:FrogJumping}, $f_e$ must be gold.
    \item []Case 3: The final vertex $\fin(f) \neq v_e$ for any $e\in T$, and $f$ is the last edge of the blue region it bounds reached by $\Phi(T)$. Then, by Lemma \ref{lem:violetsnew}, $f$ cannot be an edge of $F(T)$.
    \item []Case 4: The final vertex $\fin(f) \neq v_e$ for any $e\in T$, and $f$ is not the last edge of the blue region it bounds reached by $\Phi(T)$.  In this case, Lemma \ref{lem:violetsnew} implies that $f\in\Edg(F(T))$.
 \end{enumerate}

 
The term $\left|\{f\in\Edg(F(T)) \ | \ f=f_e \ \text{for some} \ e\in T \ \text{and} \ f \ \text{is} \ \text{blue}\}\right|$ in our computation of $w$ precisely counts the blue edges falling into Case 1, so we may replace this term with $|\Edg_{\text{away}}(T)|$. Furthermore, the term $\left|\{f\in\Edg(F(T)) \ | \ f\neq f_e \ \text{for any} \ e\in T \ \text{and} \ f \ \text{is} \ \text{blue}\}\right|$ counts the blue edges falling into Case 4.  By Lemma \ref{lem:overlap}, we can replace this term with $|\mathcal{B}| - |\Edg(T)| - m_2$. We then have 

\begin{align*}
   w &= (|\Ver(\overrightarrow{\mathcal{G}(L)})|-1) - |\mathcal{B}| + |\Edg(T)| - |\Edg_{\text{away}}(T)| + m_2 \\
    &= (|\Ver(\overrightarrow{\mathcal{G}(L)})|-1) - |\mathcal{B}| + |\Edg(T)\setminus \Edg_{\text{away}}(T)| + m_2. \\
\end{align*}

Finally, the map $f\mapsto\fin(f)$ is a bijection from $\mathcal{B}$ to $\Ver(\overrightarrow{G}(L))$. Therefore, 

$$w = |\Edg(T)\setminus \Edg_{\text{away}}(T)| + m_2 -1,$$

and we can conclude that ${\rm wt}_C(A)={\rm wt}_K(T)t^{m_2-1}$.

\qed

    \section*{Acknowledgements} The second author is grateful to Tam\'as K\'alm\'an and Alexander Postnikov for many inspiring conversations over many years. The authors are also grateful to Mario Sanchez for his interest and helpful comments about this work.

\bibliographystyle{amsplain}
\bibliography{refs}

\end{document}